\documentclass[11pt]{article} 
\usepackage{amsfonts,amsmath,latexsym,amssymb,mathrsfs,amsthm,comment}
\usepackage{caption}
\usepackage{graphicx}
\usepackage{xcolor}
\usepackage{booktabs}

\let\OLDthebibliography\thebibliography
\renewcommand\thebibliography[1]{
  \OLDthebibliography{#1}
  \setlength{\parskip}{1pt}
  \setlength{\itemsep}{0pt plus 0.0ex}
}

\def\numberlikeadb{\global\def\theequation{\thesection.\arabic{equation}}}
\numberlikeadb
\newtheorem{theorem}{Theorem}[section]
\newtheorem{lemma}[theorem]{Lemma}
\newtheorem{corollary}[theorem]{Corollary}

\newtheorem{proposition}[theorem]{Proposition}
\newtheorem{remark}[theorem]{Remark}

\usepackage{lscape}
\usepackage{caption}
\usepackage{multirow}
\allowdisplaybreaks

\begin{document}

\title{The distribution of the product of independent variance-gamma random variables
}
\author{Robert E. Gaunt\footnote{Department of Mathematics, The University of Manchester, Oxford Road, Manchester M13 9PL, UK, robert.gaunt@manchester.ac.uk; siqi.li-8@postgrad.manchester.ac.uk}\:\, and Siqi L$\mathrm{i}^{*}$}

\date{} 
\maketitle

\vspace{-5mm}

\begin{abstract} Let $X$ and $Y$ be independent variance-gamma random variables with zero location parameter; then the exact probability density function of the product $XY$ is derived. Some basic distributional properties are also derived, including formulas for the cumulative distribution function and the characteristic function, as well as asymptotic approximations for the density, tail probabilities and the quantile function. As special cases, we deduce some key distributional properties for the product of two independent asymmetric Laplace random variables as well as the product of four jointly correlated zero mean normal random variables with a particular block diagonal covariance matrix. As a by-product of our analysis, we deduce some new reduction formulas for the Meijer $G$-function.
\end{abstract}



\section{Introduction}

 The variance-gamma (VG) distribution with parameters $m > -1/2$, $0\leq|\beta|<\alpha$, $\mu \in \mathbb{R}$, denoted by $\mathrm{VG}(m,\alpha,\beta,\mu)$, has probability density function (PDF)
\begin{equation}\label{vgpdf} f_X(x) = M_{m,\alpha,\beta} \mathrm{e}^{\beta (x-\mu)}|x-\mu|^{m} K_{m}(\alpha|x-\mu|), \quad x\in \mathbb{R},
\end{equation}
where the normalising constant is given by
\[M_{m,\alpha,\beta}=\frac{\gamma^{2m+1}}{\sqrt{\pi}(2\alpha)^m \Gamma(m+1/2)},\]
with $\gamma^2=\alpha^2-\beta^2$. Here $K_m(x)$ is a modified Bessel function of the second kind, which is defined in Appendix \ref{appa}. Alternative parametrisations are given in \cite{gaunt vg,kkp01,mcc98}.
Other names include the generalized Laplace distribution \cite[Section 4.1]{kkp01}, Bessel function distribution \cite{m32} and the McKay Type II distribution \cite{ha04}. 
The VG distribution has an attractive distributional theory containing as special and limiting cases the normal, gamma and asymmetric Laplace distributions, as well as the product of correlated zero mean normal random variables amongst others \cite{vg review}.
The VG distribution has been widely used in financial modelling \cite{mcc98,madan}
and further application areas and distributional properties are given in
the review \cite{vg review} and Chapter 4 of the book \cite{kkp01}.
In this paper, we set $\mu=0$.

Let $X\sim \mathrm{VG}(m,\alpha_1,\beta_1,0)$ and $Y\sim\mathrm{VG}(n,\alpha_2,\beta_2,0)$, $m,n>-1/2$, $0\leq|\beta_i|<\alpha_i$, $i=1,2$, be independent VG random variables. In this paper, we derive an exact formula, expressed in terms of the Meijer $G$-function (defined in Appendix \ref{appa}), for the PDF of the product $XY$. Our formula generalises one of \cite{gaunt algebra} for the PDF of the product $XY$ in the case $\beta_1=\beta_2=0$ (a product of symmetric VG random variables).
We derive our formula for the PDF using the theory of Mellin transforms. A systematic account of the application of Mellin transforms to derive formulas for the PDF of products of independent random variables was first given by \cite{e48}, and this approach has since been used to find exact formulas for the PDFs of products of a number of important continuous random variables; see, for example, \cite{ks69,l67,s79,st66,product normal,wells}.  

With our formula for the PDF of the product $XY$ at hand, we are able to derive a number of other key distributional properties, including formulas for the cumulative distribution function (CDF) and the characteristic function, as well as asymptotic approximations for the PDF, tail probabilities and the quantile function, none of which were given in \cite{gaunt algebra}. Our asymptotic analysis of the PDF of the product $XY$ at the origin implies that the distribution is unimodal with mode at the origin. All these results are stated and proved in Section \ref{sec2}. 

In the case $m=1/2$, $\mu=0$ the VG distribution corresponds to the asymmetric Laplace distribution (with zero location parameter), and further setting $\beta=0$ yields the classical symmetric Laplace distribution. A comprehensive account of the distributional theory of the Laplace and asymmetric Laplace distributions is given in the excellent book \cite{kkp01}; however, there appears to be a surprising gap in the literature regarding the distributional theory of the product of independent asymmetric Laplace distributions, with, to our best knowledge,  only formulas for the PDF and CDF of the symmetric Laplace distribution (with zero location parameter) available in \cite{nad07}. We fill in this gap in Section \ref{sec3.1} by providing formulas for the PDF, CDF and characteristic function of the product of two independent asymmetric Laplace random variables (with zero location parameter). 

Also, in the case $m=\mu=0$, the VG distribution corresponds to the distribution of the product of two correlated zero mean normal random variables, which itself has numerous applications that date back to the 1930's with the works of \cite{craig,wb32}; see \cite{gaunt22} for an overview. Our formula for the PDF corresponding to this case is given in Corollary \ref{cor2.10}, which yields an exact formula for the PDF of the product $N_1N_2N_3N_4$ of four jointly correlated zero mean normal random variables $(N_1,N_2,N_3,N_4)^\intercal$ for which $(N_1,N_2)^\intercal$ and $(N_3,N_4)^\intercal$ are independent. To the best of our knowledge, this is one of the first explicit formulas in the literature for the PDF of the product of three or more zero mean correlated normal random variables; see Remark \ref{remend} for a further discussion.

The formulas obtained in this paper are expressed in terms of a number of standard special functions, all of which are defined in Appendix \ref{appa}. As a by-product of our analysis, we obtain some new reduction formulas for the Meijer $G$-function, which are collected in Section \ref{sec4}.



\vspace{3mm}

\noindent{\emph{Notation.}} To simplify formulae, we write $\lambda_{i}^{\pm}=\alpha_i\pm\beta_i$, $i=1,2$. We let $\gamma_i^2 = \alpha_i^2 - \beta_i^2$, $i=1,2$.  Finally, $\mathrm{sgn}(x)$ will denote the sign function, $\mathrm{sgn}(x)=1$ for $x>0$, $\mathrm{sgn}(0)=0$, $\mathrm{sgn}(x)=-1$ for $x<0$.

\section{The variance-gamma product distribution}\label{sec2}

\subsection{Probability density function}\label{sec2.1}
In the following theorem, we provide an explicit formula for the PDF of the product $Z=XY$ as an infinite series involving the Meijer $G$-function. 
The PDF of the product $Z$ is plotted for a range of parameter values in Figures \ref{fig1} and \ref{fig5}. Figure \ref{fig1} shows the effect of varying the shape parameters $m$ and $n$, whilst Figure \ref{fig5} demonstrates the effect of varying the skewness parameters $\beta_1$ and $ \beta_2$. In both figures, $\alpha_1=\alpha_2=1$. 

\begin{theorem}\label{thm1}
Let $m,n>-1/2$, $0\leq |\beta_i|<\alpha_i$, $i=1,2$. Suppose $X \sim \mathrm{VG}(m,\alpha_1,\beta_1,0) $ and $Y \sim \mathrm{VG}(n,\alpha_2,\beta_2,0) $ are independent, and let $Z=X  Y$ denote their product. Then, for $z\in\mathbb{R}$,
\begin{align}
f_Z(z)  =  \frac{\gamma_1^{2m+1} \gamma_2^{2n+1} }{4 \pi \alpha_1^{2m} \alpha_2^{2n} \Gamma(m+1/2) \Gamma(n+1/2)  } \sum^\infty_{i=0} \sum^{2i}_{j=0} \frac{(\mathrm{sgn}(z))^j}{j!(2i-j)!} \bigg(\frac{2\beta_1}{\alpha_1}\bigg)^j\bigg(\frac{2\beta_2}{\alpha_2}\bigg)^{2i-j} &\nonumber
\\   \times G^{4,0}_{0,4} \bigg( \frac{\alpha_1^2 \alpha_2^2}{16}z^2 \bigg| { - \atop \frac{j}{2}, i-\frac{j}{2},m+\frac{j}{2},n+i-\frac{j}{2}  }   \bigg),& \quad z\in\mathbb{R}. \label{eq:3}
\end{align}
\end{theorem}


\captionsetup{font=footnotesize}

\begin{figure}[h]\label{fig1}
\centering
\includegraphics[scale=0.47]{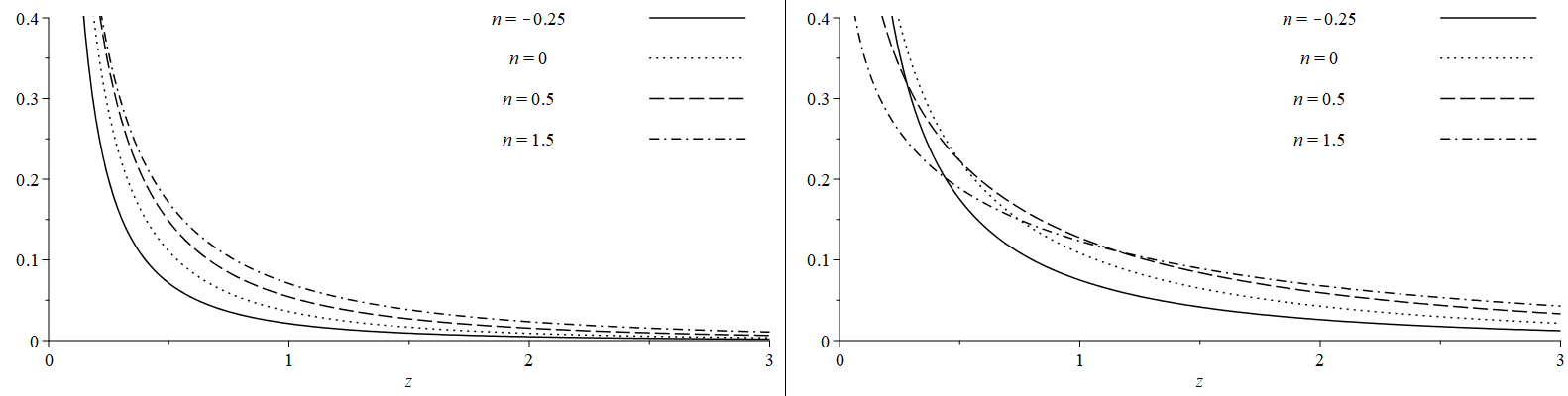}
\caption{$\beta_1=\beta_2=0,(a):m=-0.25 ;(b):m=2$}
  \label{fig1}
  \includegraphics[scale=0.8]{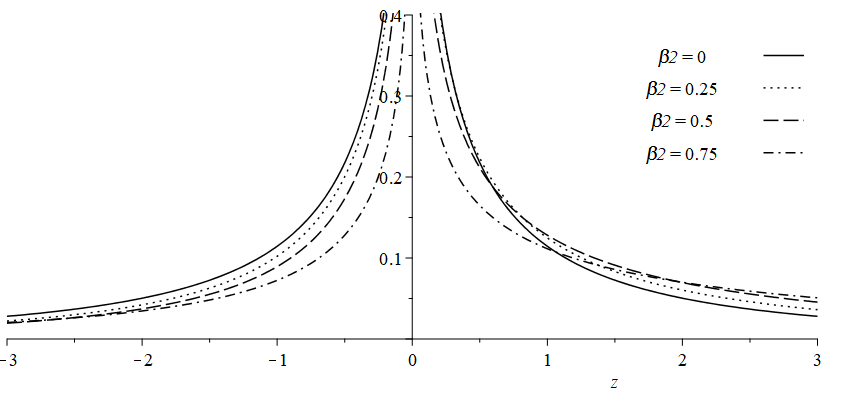}
\caption{$m=n=0.5$, $\beta_1=0.5$}
  \label{fig5}
\end{figure}

\begin{remark} When both skewness parameters $\beta_1$ and $\beta_2$ are zero, the PDF $f_Z(z)$ reduces to a single Meijer $G$-function:
\begin{align}\label{bb00}
f_Z(z) & =  \frac{\alpha_1 \alpha_2 }{4 \pi  \Gamma(m+1/2) \Gamma(n+1/2) }     G^{4,0}_{0,4} \bigg( \frac{\alpha_1^2 \alpha_2^2}{16}z^2 \bigg| { - \atop 0, 0,m,n } \bigg), \quad z\in\mathbb{R},
\end{align}
which is in agreement with equation (37) of \cite{gaunt algebra}. 
If one of the skewness parameters is equal to zero, then the PDF (\ref{eq:3}) reduces to a single infinite series. Without loss of generality, let $\beta_2=0$. Then, for $z\in\mathbb{R}$,
\begin{align}
f_Z(z) & =  \frac{\alpha_2 \gamma_1^{2m+1} }{4 \pi\alpha_1^{2m}  \Gamma(m+1/2) \Gamma(n+1/2)  }  \sum^\infty_{k=0} \frac{1}{(2k)!}\bigg(\frac{2\beta_1}{\alpha_1}\bigg)^{2k}  G^{4,0}_{0,4} \bigg( \frac{\alpha_1^2 \alpha_2^2}{16}z^2 \bigg| { - \atop k, 0,m+k,n  }  \bigg).
 \nonumber
\end{align}
We observe that the PDF $f_Z(z)$ is symmetric about the origin if $\beta_1=0$ or $\beta_2=0$.

We remark that the exact PDF of the ratio $X/Y$, where $X\sim \mathrm{VG}(m,\alpha_1,\beta_1,0)$ and $Y\sim\mathrm{VG}(n,\alpha_2,\beta_2,0)$ are independent, is expressed as a double sum involving the Gaussian hypergeometric function when both $\beta_1$ and $\beta_2$ are non-zero \cite{gl23}, and the PDF simplifies to a single infinite series of hypergeometric functions if one of the skewness parameters is zero \cite{gl23}, and to just a single hypergeometric function if $\beta_1=\beta_2=0$ \cite{nk06}. The increase in complexity of the PDF of the product $XY$ for non-zero skewness parameters thus mirrors that of the ratio $X/Y$. A similar behaviour also occurs for the product of two correlated normal random variables as one transitions from zero to non-zero means \cite{cui}.
\end{remark}

\noindent{\emph{Proof of Theorem \ref{thm1}.}}
We first consider the case $z>0$. For a random variable $V$, we define its positive part $V^+$ by $V^+  = \max\{V,0\}$ and its negative part $V^-$ by $V^-  =\max\{-V,0\}$. Utilising the Taylor expansion of the exponential function and evaluating the resulting integrals with formula (\ref{eq:32}), we derive the Mellin transform for $X^+$:
\begin{align}
\mathcal{M}_{ f_{X^+} }(s) & = \int^\infty_0 x^{s} f_{X^+}(x) \, \mathrm{d}x=M_{m,\alpha_1,\beta_1}\sum_{i=0}^\infty\frac{\beta_1^i}{i!}\int_0^\infty x^{s+m+i}K_m(\alpha_1 x)\,\mathrm{d}x\nonumber
\\  & = \frac{\gamma_1^{2m+1}}{2 \sqrt{\pi} \alpha_1^{2m+1}\Gamma(m+1/2)  } \sum^\infty_{i=0} \frac{ 2^i \beta_1^i}{ \alpha_1^i i!} \frac{2^{s}}{\alpha_1^s} \Gamma\Big(\frac{s}{2}+\frac{i+1}{2}\Big) \Gamma\Big(\frac{s}{2}+m+\frac{i+1}{2}\Big).\label{111}
\end{align} 
Similarly, we obtain the Mellin transform for $X^-$:
\begin{equation}\label{222}
\mathcal{M}_{ f_{X^-}  }(s)  = \frac{\gamma_1^{2m+1}}{2\sqrt{\pi} \alpha_1^{2m+1} \Gamma(m+1/2) } \sum^\infty_{i=0}  \frac{2^i \beta_1^i (-1)^i}{ \alpha_1^i i!} \frac{2^{s}}{\alpha_1^s} \Gamma\Big(\frac{s}{2}+\frac{i+1}{2}\Big) \Gamma\Big(\frac{s}{2}+m+\frac{i+1}{2}\Big).
\end{equation} 
The Mellin transform of $Z^+$ is then given by 
\begin{align}\label{333}
\mathcal{M}_{Z^+}(s) = \mathcal{M}_{X^+}(s) \mathcal{M}_{Y^+}(s) + \mathcal{M}_{X^-}(s) \mathcal{M}_{Y^-}(s). 
\end{align}
Taking the inverse Mellin transform yields the PDF of $Z^+$, and we need to find:
\begin{align}
 & \mathcal{M}^{-1} \bigg(  \frac{4^{s}}{\alpha_1^s \alpha_2^s } \Gamma\Big(\frac{s}{2}+\frac{i+1}{2}\Big) \Gamma\Big(\frac{s}{2}+m+\frac{i+1}{2}\Big)  \Gamma\Big(\frac{s}{2}+\frac{j+1}{2}\Big) \Gamma\Big(\frac{s}{2}+n+\frac{j+1}{2}\Big) \bigg) \nonumber
 \\  = & \frac{1}{2 \pi \mathrm{i}}  \int^{c+\mathrm{i}\infty}_{c-\mathrm{i}\infty} z^{-s-1}  \frac{4^{s}}{\alpha_1^s \alpha_2^s } \Gamma\Big(\frac{s}{2}+\frac{i+1}{2}\Big) \Gamma\Big(\frac{s}{2}+m+\frac{i+1}{2}\Big)  \Gamma\Big(\frac{s}{2}+\frac{j+1}{2}\Big) \Gamma\Big(\frac{s}{2}+n+\frac{j+1}{2}\Big)\, \mathrm{d}s \nonumber
  \\ = & \frac{\alpha_1 \alpha_2}{2} G^{4,0}_{0,4} \bigg( \frac{\alpha_1^2 \alpha_2^2}{16}z^2 \bigg| { - \atop \frac{i}{2}, \frac{j}{2},m+\frac{i}{2},n+\frac{j}{2}  }   \bigg), \label{eq:33}
\end{align}
where $c>0$ and we evaluated the contour integral using a change of variables and the contour integral definition (\ref{mdef}) of the Meijer $G$-function. On combining (\ref{111}), (\ref{222}), (\ref{333}) and (\ref{eq:33}) we deduce a formula for the PDF of $Z$ for $z>0$. Arguing similarly we can obtain a formula for the PDF of $Z$ for $z<0$, and overall we get that, for $z\in\mathbb{R}$,
\begin{align}
f_Z(z)  =  \frac{\gamma_1^{2m+1} \gamma_2^{2n+1} }{4 \pi  \alpha_1^{2m} \alpha_2^{2n}\Gamma(m+1/2) \Gamma(n+1/2)  } &\sum^\infty_{i=0} \sum^\infty_{j=0} \frac{a_{ij}(\mathrm{sgn}(z))^i}{i!j!} \bigg(\frac{2\beta_1}{\alpha_1}\bigg)^i\bigg(\frac{2\beta_2}{\alpha_2}\bigg)^j \nonumber
\\ &  \times G^{4,0}_{0,4} \bigg( \frac{\alpha_1^2 \alpha_2^2}{16}z^2 \bigg| { - \atop \frac{i}{2}, \frac{j}{2},m+\frac{i}{2},n+\frac{j}{2}  }   \bigg),  \label{eq:3nn}
\end{align}
where $a_{ij}=(1+(-1)^{i+j})/2$, $i,j\geq0$, so that $a_{ij}=1$ if $i+j$ is even, and $a_{ij}=0$ if $i+j$ is odd. The formula (\ref{eq:3nn}) can be easily rewritten as (\ref{eq:3}), and this completes the proof.
\hfill $\Box$

\vspace{3mm}

\begin{remark}
As noted by one of the reviewers, the series representations (\ref{111}) and (\ref{222}) of $\mathcal{M}_{ f_{X^+}  }(s)$ and $\mathcal{M}_{ f_{X^-}  }(s)$, respectively, can also be represented in terms of the Fox--Wright function (defined in Appendix \ref{appa}):
\begin{align*}
\mathcal{M}_{ f_{X^{\pm}}  }(s)=\frac{2^{s-1}\gamma_1^{2m+1}\Gamma((s+1)/2)\Gamma((s+1)/2+m)}{\sqrt{\pi}\alpha_1^{2m+s+1}\Gamma(m+1/2)}\,{}_{2}\Psi_0\bigg({(\frac{s+1}{2},\frac{1}{2}),(\frac{s+1}{2}+m,\frac{1}{2}) \atop -}\,\bigg|\,\pm\frac{2\beta_1}{\alpha_1} \bigg).   
\end{align*}
\end{remark}



Using special properties of the modified Bessel function of the second kind, simpler formulas for the PDF of $Z=XY$ can be obtained. We illustrate this in the following Proposition \ref{cor.2}. Using reduction formulas for the Meijer $G$-function it is also possible to obtain simpler formulas for the PDF; see Section \ref{sec3} in which we provide simplified formulas for certain special cases of the VG product distribution.

\begin{proposition}\label{cor.2}
Suppose that $m-1/2\geq0$ and $n-1/2\geq0$ are non-negative integers. Then,  (\ref{eq:3}) can be simplified to a finite double sum expressed in terms of the modified Bessel function of the second kind:
\begin{align} \label{eq:17}
f_Z(z)  & = \frac{\gamma_1^{2m+1} \gamma_2^{2n+1}}{ 2^{m+n} \alpha_1^{m+1/2} \alpha_2^{n+1/2} (m-1/2)!(n-1/2)! } \sum^{m-\frac{1}{2}}_{i=0} \sum^{n-\frac{1}{2}}_{j=0}  \frac{(m-1/2+i)!}{i!(m-1/2-i)!}  \frac{(n-1/2+j)!}{j!(n-1/2-j)!}  \nonumber
\\ & \qquad \times \frac{|z|^{n-1/2-j}}{(2\alpha_1)^i (2\alpha_2)^j}  \bigg\{ \bigg(\frac{\alpha_2|z|-\beta_2z}{\alpha_1-\beta_1}\bigg)^{\frac{1}{2}(m-n-i+j)} K_{m-n-i+j}\big(2\sqrt{\alpha_1-\beta_1}\sqrt{\alpha_2 |z|-\beta_2 z}\big) \nonumber
\\ & \qquad +  \bigg(\frac{\alpha_2 |z|+\beta_2 z}{\alpha_1+\beta_1}\bigg)^{\frac{1}{2}(m-n-i+j)} K_{m-n-i+j} \big(2\sqrt{\alpha_1+\beta_1}\sqrt{\alpha_2 |z|+\beta_2 z}\big)   \bigg\}, \quad z\in\mathbb{R}. 
\end{align}
\end{proposition}

\begin{proof} The modified Bessel function of the second kind takes on an elementary form (\ref{special}) in this case. Consequently, the PDF of $Z = XY$ can be expressed as
\begin{align*}
f_Z(z) & = \int^\infty_{-\infty} f_X(x) f_Y\bigg(\frac{z}{x}\bigg) \frac{1}{|x|} \,\mathrm{d}x
\\ & = \frac{\gamma_1^{2m+1} \gamma_2^{2n+1}}{(2\alpha_1)^{m+1/2} (2\alpha_2)^{n+1/2} \Gamma(m+1/2) \Gamma(n+1/2) } \sum^{m-\frac{1}{2}}_{i=0} \sum^{n-\frac{1}{2}}_{j=0}  \frac{(m-1/2+i)!}{i!(m-1/2-i)!}  \frac{(n-1/2+j)!}{j!(n-1/2-j)!}  
\\ & \qquad \times \frac{|z|^{n-1/2-j}}{(2\alpha_1)^i (2\alpha_2)^j} \int^\infty_{-\infty}  \mathrm{e}^{\beta_1 x-\alpha_1 |x| +\beta_2 z/x - \alpha_2 |z|/|x| } |x|^{m-n+j-i-1}      \,\mathrm{d}x.
\end{align*}
Evaluating the integrals using (\ref{kdefn}) yields (\ref{eq:17}). 
\end{proof}

\subsection{Cumulative distribution function}\label{sec2.2}

Let $F_Z(z)=\mathbb{P}(Z\leq z)$ denote the CDF of the product $Z=XY$. In the following theorem, we provide an exact formula for the CDF of $Z$ for the case of independent symmetric VG random variables ($\beta_1=\beta_2=0$).


\begin{theorem} \label{thm3} Let $\beta_1=\beta_2=0$ and $m,n>-1/2$, $\alpha_1,\alpha_2>0$. Then, for $z\in\mathbb{R}$,
\begin{align}\label{fgh}F_Z(z)&= \frac{1}{2} + \frac{\alpha_1 \alpha_2 z}{8 \pi \Gamma(m+1/2) \Gamma(n+1/2) } G^{4,1}_{1,5}\bigg( \frac{\alpha_1^2 \alpha_2^2}{16}z^2 \bigg|\,{ \frac{1}{2} \atop 0,0,m,n,-\frac{1}{2}}\bigg)\\
\label{near}&=\frac{1}{2}+\frac{\mathrm{sgn}(z)}{2 \pi \Gamma(m+1/2) \Gamma(n+1/2) } G^{4,1}_{1,5}\bigg( \frac{\alpha_1^2 \alpha_2^2}{16}z^2 \bigg|\,{ 1 \atop \frac{1}{2},\frac{1}{2},m+\frac{1}{2},n+\frac{1}{2},0}\bigg).
\end{align}
\end{theorem}

\begin{proof}For ease of exposition, we set $\alpha_1=\alpha_2=1$, with general case following from a simple rescaling. As the PDF (as given by (\ref{bb00}) in this case) is symmetric about the origin when $\beta_1=\beta_2=0$, we also only consider the case $z>0$. Letting $1/C_{m,n}=4\pi\Gamma(m+1/2)\Gamma(n+1/2)$, we have, for $z>0$, 
\begin{align}
F_Z(z)&=\frac{1}{2}+C_{m,n}\int_0^zG^{4,0}_{0,4} \bigg( \frac{1}{16}y^2 \bigg| { - \atop 0, 0,m,n } \bigg)\,\mathrm{d}y\nonumber\\
&=\frac{1}{2}+2C_{m,n}\int_0^{z^2/16}u^{-1/2}G^{4,0}_{0,4} \bigg( u \bigg| { - \atop 0, 0,m,n } \bigg)\,\mathrm{d}u\nonumber\\
\label{zzzza}&=\frac{1}{2}+2C_{m,n}\bigg[u^{1/2}G^{4,1}_{1,5} \bigg( u \bigg| {\frac{1}{2}  \atop 0, 0,m,n,-\frac{1}{2} } \bigg)\bigg]_0^{z^2/16}\\
&=\frac{1}{2} + \frac{C_{m,n}}{2}z \,G^{4,1}_{1,5}\bigg( \frac{1}{16}z^2 \bigg|\,{ \frac{1}{2} \atop 0,0,m,n,-\frac{1}{2}}\bigg),\nonumber
\end{align}
where in the third step we used the indefinite integral formula (\ref{mint}) and in the final step we used that the Meijer $G$-function in (\ref{zzzza}) evaluated at $u=0$ is equal to zero, which is easily seen from the contour integral representation (\ref{mdef}) of the $G$-function. Finally, we deduce (\ref{near}) from (\ref{fgh}) using the relation (\ref{meijergidentity}). 
 \end{proof}

When $m-1/2\geq0$ and $n-1/2\geq0$ are non-negative integers, a finite double sum representation of the CDF can be given, even for non-zero skewness parameters. For $\mu\geq\nu>-1/2$, let
\begin{align}\label{gmn2} G_{\mu,\nu}(x)=x\big(K_{\nu}( x)\tilde{t}_{\mu-1,\nu-1}( x)+K_{\nu-1}( x)\tilde{t}_{\mu,\nu}( x)\big),\quad
\tilde{G}_{\mu,\nu}(x)=1-G_{\mu,\nu}(x),
\end{align}
where $\tilde{t}_{\mu,\nu}(x)$ is a normalisation of the modified Lommel function of the first kind $t_{\mu,\nu}(x)$, which is defined in Appendix \ref{appa}. 

\begin{proposition}\label{pmn}Suppose that $m-1/2\geq0$ and $n-1/2\geq0$ are non-negative integers and that $0\leq |\beta_i|<\alpha_i$, $i=1,2$. Let $p=(m-n-i+j)/2$ and $q=(m+n-i-j)/2$. Then, for $z>0$,
\begin{align}
F_Z(z)   = 1-&\frac{\gamma_1^{2m+1} \gamma_2^{2n+1}}{  (2\alpha_1)^{m+1/2} (2\alpha_2)^{n+1/2} (m-1/2)!(n-1/2)! } \sum^{m-\frac{1}{2}}_{i=0} \sum^{n-\frac{1}{2}}_{j=0}  \frac{(m-1/2+i)!}{i!(2\alpha_1)^i}  \frac{(n-1/2+j)!}{j!(2\alpha_2)^j}  \nonumber
\\ & \quad \times   \bigg\{ \bigg(\frac{\lambda_2^-}{\lambda_1^-}\bigg)^{p} \frac{\tilde{G}_{2q,2p}\big(2(\lambda_1^-\lambda_2^-)^{1/2}\sqrt{z}\big)}{(\lambda_1^-\lambda_2^-)^{q+1/2}} +\bigg(\frac{\lambda_2^+}{\lambda_1^+}\bigg)^{p} \frac{\tilde{G}_{2q,2p}\big(2(\lambda_1^+\lambda_2^+)^{1/2}\sqrt{z}\big)}{(\lambda_1^+\lambda_2^+)^{q+1/2}}\bigg\},  \label{cdf5}
\end{align}    
and, for $z<0$,
\begin{align}
F_Z(z)  & = \frac{\gamma_1^{2m+1} \gamma_2^{2n+1}}{  (2\alpha_1)^{m+1/2} (2\alpha_2)^{n+1/2} (m-1/2)!(n-1/2)! } \sum^{m-\frac{1}{2}}_{i=0} \sum^{n-\frac{1}{2}}_{j=0}  \frac{(m-1/2+i)!}{i!(2\alpha_1)^i}  \frac{(n-1/2+j)!}{j!(2\alpha_2)^j}  \nonumber
\\ & \qquad \times  \bigg\{ \bigg(\frac{\lambda_2^+}{\lambda_1^-}\bigg)^{p} \frac{\tilde{G}_{2q,2p}\big(2(\lambda_1^-\lambda_2^+)^{1/2}\sqrt{-z}\big)}{(\lambda_1^-\lambda_2^+)^{q+1/2}}  +\bigg(\frac{\lambda_2^-}{\lambda_1^+}\bigg)^{p} \frac{\tilde{G}_{2q,2p}\big(2(\lambda_1^+\lambda_2^-)^{1/2}\sqrt{-z}\big)}{(\lambda_1^+\lambda_2^-)^{q+1/2}}\bigg\}.  \nonumber
\end{align} 
Suppose now that $\beta_1=\beta_2=0$. Then, for $z\in\mathbb{R}$,
\begin{align}
F_Z(z)=\frac{1}{2}+\frac{\alpha_1^{m}\alpha_2^{n}\,\mathrm{sgn}(z)}{2^{m+n}(m-1/2)!(n-1/2)!} \sum^{m-\frac{1}{2}}_{i=0} \sum^{n-\frac{1}{2}}_{j=0} &  \frac{(m-1/2+i)!}{i!(2\alpha_1)^i}  \frac{(n-1/2+j)!}{j!(2\alpha_2)^j}\nonumber\\
\label{73}&\times\bigg(\frac{\alpha_2}{\alpha_1}\bigg)^p\frac{G_{2q,2p}(2\sqrt{\alpha_1\alpha_2|z|})}{(\alpha_1\alpha_2)^{q}}.
\end{align}
\end{proposition} 

\begin{proof}We need to compute $F_Z(z)=\int_{-\infty}^z f_Z(y)\,\mathrm{d}y$, where $f_Z(z)$ is given by equation (\ref{eq:17}). We consider the case $z>0$; the case $z<0$ is similar and so omitted. Note that $F_Z(z)=1-\int_z^\infty f_Z(y)\,\mathrm{d}y$, and evaluating the integral using (\ref{kint2}) yields the formula (\ref{cdf5}). Suppose now that $\beta_1=\beta_2=0$, so that the PDF $f_Z(z)$ is symmetric about the origin. We obtain (\ref{73}) by calculating $F_Z(z)=1/2+\mathrm{sgn}(z)\int_0^{|z|}f_Z(y)\,\mathrm{d}y$ using the integral formula (\ref{37}).
\end{proof}


A simple formula can be given for the probability $\mathbb{P}(Z\leq0)$. We will make use of a recent result of \cite{gaunt24}, which states that, for $X\sim\mathrm{VG}(m,\alpha,\beta,0)$,
\begin{equation}\label{asdfg}
\mathbb{P}(X\leq0)=P_{m,\alpha,\beta},    
\end{equation}
where
\begin{equation*}P_{m,\alpha,\beta}=\frac{1}{2}-\frac{\Gamma(m+1)}{\sqrt{\pi}\Gamma(m+1/2)}\frac{\beta}{\alpha}\bigg(1-\frac{\beta^2}{\alpha^2}\bigg)^{m+1/2}{}_2F_1\bigg(1,m+1;\frac{3}{2};\frac{\beta^2}{\alpha^2}\bigg),
\end{equation*}
and the Gaussian hypergeometric function is defined in Appendix \ref{appa}.
\begin{proposition}Let $m,n>-1/2$, $0\leq |\beta_i|<\alpha_i$, $i=1,2$. Then
\begin{align*}
\mathbb{P}(Z\leq0)=P_{m,\alpha_1,\beta_1}+P_{n,\alpha_2,\beta_2}-2P_{m,\alpha_1,\beta_1}P_{n,\alpha_2,\beta_2}.  
\end{align*}
\end{proposition} 

\begin{proof}A simple calculation that exploits the fact that $X$ and $Y$ are independent shows that $\mathbb{P}(Z\leq0)=\mathbb{P}(X\leq0)+\mathbb{P}(Y\leq0)-2\mathbb{P}(X\leq0)\mathbb{P}(Y\leq0)$, and the result now follows from the formula (\ref{asdfg}).    
\end{proof}

\begin{remark}We used \emph{Mathematica} to calculate $\mathbb{P}(Z\leq0)$ for the case $\alpha_1=\alpha_2=1$, for a range of values of $m,n$ and $\beta_1,\beta_2$; the results are reported in Table \ref{table1}. Only positive values of $\beta_1$ and $\beta_2$ are considered due to the fact that if $X\sim \mathrm{VG}(m,1,\beta,0)$, then $-X\sim\mathrm{VG}(m,1,-\beta,0)$ (see \cite[Section 2.1]{vg review}). From Table \ref{table1}, we observe that the probability $\mathbb{P}(Z\leq0)$ decreases as the skewness parameters $\beta_1,\beta_2$ increase and as the shape parameters $m,n$ increase. 
\end{remark}

\begin{table}[h]
  \centering
  \caption{$\mathbb{P}(Z\leq0)$ for $Z=XY$, where $X\sim\mathrm{VG}(m,1,\beta_1,0)$ and $Y\sim\mathrm{VG}(n,1,\beta_2,0)$ are independent.
  }
\label{table1}
\normalsize{
\begin{tabular}{l*{6}{c}}
\hline
& \multicolumn{6}{c}{$(m,n)$} \\
\cmidrule(lr){2-7}
$(\beta_1,\beta_2)$ & (0,0) & (0,1.5) & (0,3) & (1.5,0) & (1.5,1.5) & (1.5,3) \\
\hline
        (0.25,0.25) & 0.4871 & 0.4705 & 0.4611 & 0.4705 & 0.4236  & 0.4112  \\ 
        (0.25,0.50) & 0.4732 & 0.4447 & 0.4333 & 0.4388 & 0.3738 & 0.3477  \\ 
        (0.25,0.75) & 0.4566 & 0.4265 & 0.4212 & 0.4009 & 0.3322 & 0.3201 \\ \hline
        (0.50,0.25) & 0.4732 & 0.4388 & 0.4194 & 0.4447 & 0.3738 & 0.3338 \\
        (0.50,0.50) & 0.4444 & 0.3854 & 0.3617 & 0.3854 & 0.2637  & 0.2148 \\ 
        (0.50,0.75) & 0.4100 & 0.3477 & 0.3367 & 0.3144 & 0.1858 & 0.1631 \\ \hline 
        (0.75,0.25) & 0.4566 & 0.4009 & 0.3695 & 0.4265 & 0.3322 & 0.2790 \\
        (0.75,0.50) & 0.4100 & 0.3144 & 0.2761 & 0.3477 & 0.1858  & 0.1209  \\ 
        (0.75,0.75) & 0.3543 & 0.2533 & 0.2354 & 0.2533 & 0.0822 & 0.0521 \\ \hline
    \end{tabular}}
\end{table}

\subsection{Characteristic function}

Let $\varphi_Z(t)=\mathbb{E}[\mathrm{e}^{\mathrm{i}tZ}]$ denote the characteristic function of the product $Z=XY$. In the following theorem, we provide an exact formula for the characteristic function for the case of independent symmetric VG random variables ($\beta_1=\beta_2=0$). We note that the moment generating function $M_Z(t)=\mathbb{E}[\mathrm{e}^{tZ}]$ is not defined for $t\not=0$, which is easily seen from the limiting forms (\ref{finfty1}) and (\ref{finfty2}).

\begin{theorem}Let $\beta_1=\beta_2=0$ and $m,n>-1/2$, $\alpha_1,\alpha_2>0$. Then, for $t\in\mathbb{R}$,
\begin{align}
\label{char}   \varphi_Z(t)&= \frac{\alpha_1 \alpha_2 }{2 \sqrt{\pi}  \Gamma(m+1/2) \Gamma(n+1/2) } |t|^{-1} G^{3,1}_{1,3} \bigg( \frac{\alpha_1^2 \alpha_2^2}{4t^2} \bigg| { \frac{1}{2} \atop  0,m,n } \bigg)\\
\label{near2}&= \frac{1}{ \sqrt{\pi}  \Gamma(m+1/2) \Gamma(n+1/2) } G^{3,1}_{1,3} \bigg( \frac{\alpha_1^2 \alpha_2^2}{4t^2} \bigg| { 1 \atop  \frac{1}{2},m+\frac{1}{2},n+\frac{1}{2} } \bigg).
\end{align}
\end{theorem}

\begin{proof}
Since the PDF (\ref{bb00}) of $Z$ is symmetric about the origin we have that 
\begin{align*}
\varphi_Z(t)&=\mathbb{E}[\cos(tZ)]=2\int_0^\infty  \frac{\alpha_1 \alpha_2 }{4 \pi  \Gamma(m+1/2) \Gamma(n+1/2) }  \cos(tz)   G^{4,0}_{0,4} \bigg( \frac{\alpha_1^2 \alpha_2^2}{16}z^2 \bigg| { - \atop 0, 0,m,n } \bigg)\,\mathrm{d}z\\
&=\frac{\alpha_1 \alpha_2 }{2 \pi  \Gamma(m+1/2) \Gamma(n+1/2) }  \cdot\frac{\sqrt{\pi}}{|t|} G^{4,1}_{2,4} \bigg( \frac{\alpha_1^2 \alpha_2^2}{4t^2} \bigg| { \frac{1}{2} ,0\atop 0, 0,m,n } \bigg)\\
&=\frac{\alpha_1 \alpha_2 }{2 \sqrt{\pi}  \Gamma(m+1/2) \Gamma(n+1/2) } |t|^{-1} G^{3,1}_{1,3} \bigg( \frac{\alpha_1^2 \alpha_2^2}{4t^2} \bigg| { \frac{1}{2} \atop  0,m,n } \bigg),
\end{align*}
where we evaluated the integral using formula 5.6.3(18) of \cite{luke} and obtained the final equality using (\ref{lukeformula}). Finally, we obtain (\ref{near2}) from (\ref{char}) using the relation (\ref{meijergidentity}).
\end{proof}


When $m-1/2\geq0$ and $n-1/2\geq0$ are non-negative integers, a finite double sum representation of the characteristic function can be given, even for non-zero skewness parameters. The formula is expressed in terms of the Whittaker function, which is defined in Appendix \ref{appa}.

\begin{proposition}\label{fghj}Suppose $m-1/2\geq0$ and $n-1/2\geq0$ are non-negative integers and that $0\leq |\beta_i|<\alpha_i$, $i=1,2$. Let $p=(m-n-i+j)/2$ and $q=(m+n-i-j)/2$. Then, for $t\in\mathbb{R}$,
\begin{align}
\varphi_Z(t)&=\frac{\gamma_1^{2m+1} \gamma_2^{2n+1}}{ (2\alpha_1)^{m+1/2} (2\alpha_2)^{n+1/2} (m-1/2)!(n-1/2)! } \sum^{m-\frac{1}{2}}_{i=0} \sum^{n-\frac{1}{2}}_{j=0}  \frac{(m-1/2+i)!}{i!(2\alpha_1)^i}  \frac{(n-1/2+j)!}{j!(2\alpha_2)^j}  \nonumber
\\ & \qquad \times   \bigg\{ \bigg(\frac{\lambda_2^-}{\lambda_1^-}\bigg)^{p}\frac{(-\mathrm{i}t)^{-q}}{(\lambda_1^-\lambda_2^-)^{1/2}}\exp\bigg(\frac{\mathrm{i}\lambda_1^-\lambda_2^-}{2t}\bigg)W_{-q,p}\bigg(\frac{\mathrm{i}\lambda_1^-\lambda_2^-}{t}\bigg)  \nonumber\\
&\qquad+\bigg(\frac{\lambda_2^+}{\lambda_1^+}\bigg)^{p}\frac{(-\mathrm{i}t)^{-q}}{(\lambda_1^+\lambda_2^+)^{1/2}}\exp\bigg(\frac{\mathrm{i}\lambda_1^+\lambda_2^+}{2t}\bigg)W_{-q,p}\bigg(\frac{\mathrm{i}\lambda_1^+\lambda_2^+}{t}\bigg) \nonumber \\
&\qquad+\bigg(\frac{\lambda_2^+}{\lambda_1^-}\bigg)^{p}\frac{(\mathrm{i}t)^{-q}}{(\lambda_1^-\lambda_2^+)^{1/2}}\exp\bigg(-\frac{\mathrm{i}\lambda_1^-\lambda_2^+}{2t}\bigg)W_{-q,p}\bigg(-\frac{\mathrm{i}\lambda_1^-\lambda_2^+}{t}\bigg)  \nonumber\\
\label{pok}&\qquad+\bigg(\frac{\lambda_2^-}{\lambda_1^+}\bigg)^{p}\frac{(\mathrm{i}t)^{-q}}{(\lambda_1^+\lambda_2^-)^{1/2}}\exp\bigg(-\frac{\mathrm{i}\lambda_1^+\lambda_2^-}{2t}\bigg)W_{-q,p}\bigg(-\frac{\mathrm{i}\lambda_1^+\lambda_2^-}{t}\bigg)  \bigg\}.
\end{align}
\end{proposition}

\begin{proof} We calculate $\varphi_Z(t)=\mathbb{E}[\mathrm{e}^{\mathrm{i}tZ}]=\int_{-\infty}^\infty \mathrm{e}^{\mathrm{i}tz}f_Z(z)\,\mathrm{d}z$ using the formula (\ref{eq:17}) for the PDF and then compute the resulting integral using the integral formula (\ref{thnaaa}).
\end{proof}

\subsection{Asymptotic behaviour of the density, tail probabilities and quantile function}

The representations of the PDF and CDF from Sections \ref{sec2.1} and \ref{sec2.2} are
rather complicated and difficult to parse on first inspection. Some insight can be gained from
the following asymptotic approximations.

\begin{proposition}\label{cor2.7}Let $m,n>-1/2$, $0\leq|\beta_i|<\alpha_i$, $i=1,2$. Let $\bar{F}_Z(z)=\mathbb{P}(Z>z)$. Then, as $z \rightarrow \infty$,
\begin{align}
\bar{F}_{Z}(z) & \sim  \frac{\sqrt{\pi} \gamma_1^{2m+1}\gamma_2^{2n+1} z^{\frac{1}{4}(2m+2n-1)}}{(2\alpha_1)^{m+1/2}(2\alpha_2)^{n+1/2}\Gamma(m+1/2)\Gamma(n+1/2)} \bigg( (\lambda_1^-)^{\frac{1}{4}(2n-2m-3)} (\lambda_2^-)^{\frac{1}{4}(2m-2n-3)} 
\nonumber\\ 
\label{limf1}& \qquad \times  \mathrm{e}^{-2\sqrt{ \lambda_1^- \lambda_2^- z}} + (\lambda_1^+)^{\frac{1}{4}(2n-2m-3)} (\lambda_2^+)^{\frac{1}{4}(2m-2n-3)}  \mathrm{e}^{-2\sqrt{\lambda_1^+ \lambda_2^+ z}} \bigg),
\end{align}
and, as $z \rightarrow -\infty$,
\begin{align}
F_{Z}(z) & \sim  \frac{\sqrt{\pi} \gamma_1^{2m+1}\gamma_2^{2n+1} (-z)^{\frac{1}{4}(2m+2n-1)}}{(2\alpha_1)^{m+1/2}(2\alpha_2)^{n+\frac{1}{2}}\Gamma(m+1/2)\Gamma(n+1/2)} \bigg( (\lambda_1^-)^{\frac{1}{4}(2n-2m-3)} (\lambda_2^+)^{\frac{1}{4}(2m-2n-3)} \nonumber
\\ 
\label{limf2}& \qquad \times  \mathrm{e}^{-2\sqrt{ -\lambda_1^- \lambda_2^+ z}} + (\lambda_1^+)^{\frac{1}{4}(2n-2m-3)} (\lambda_2^-)^{\frac{1}{4}(2m-2n-3)}    \mathrm{e}^{-2\sqrt{-\lambda_1^+ \lambda_2^- z}} \bigg).
\end{align}
\end{proposition}

\begin{corollary}\label{prop2.5} Let $m,n>-1/2$, $0\leq|\beta_i|<\alpha_i$, $i=1,2$. Then, as $z\rightarrow \infty$,
\begin{align}
f_Z(z) & \sim  \frac{\sqrt{\pi} \gamma_1^{2m+1}\gamma_2^{2n+1} z^{\frac{1}{4}(2m+2n-3)}}{(2\alpha_1)^{m+1/2}(2\alpha_2)^{n+1/2}\Gamma(m+1/2)\Gamma(n+1/2)} \bigg( (\lambda_1^-)^{\frac{1}{4}(2n-2m-1)} (\lambda_2^-)^{\frac{1}{4}(2m-2n-1)}  \nonumber
\\ 
\label{finfty1}& \qquad \times\mathrm{e}^{-2\sqrt{ \lambda_1^- \lambda_2^- z}} + (\lambda_1^+)^{\frac{1}{4}(2n-2m-1)} (\lambda_2^+)^{\frac{1}{4}(2m-2n-1)}  \mathrm{e}^{-2\sqrt{\lambda_1^+ \lambda_2^+ z}} \bigg),
\end{align}
and, as $z\rightarrow -\infty$,
\begin{align}
f_{Z}(z) & \sim  \frac{\sqrt{\pi} \gamma_1^{2m+1}\gamma_2^{2n+1} (-z)^{\frac{1}{4}(2m+2n-3)}}{(2\alpha_1)^{m+1/2}(2\alpha_2)^{n+1/2}\Gamma(m+1/2)\Gamma(n+\frac{1}{2})} \bigg( (\lambda_1^-)^{\frac{1}{4}(2n-2m-1)} (\lambda_2^+)^{\frac{1}{4}(2m-2n-1)} \nonumber
\\ 
\label{finfty2}& \qquad \times  \mathrm{e}^{-2\sqrt{ -\lambda_1^- \lambda_2^+ z}} + (\lambda_1^+)^{\frac{1}{4}(2n-2m-1)} (\lambda_2^-)^{\frac{1}{4}(2m-2n-1)}    \mathrm{e}^{-2\sqrt{-\lambda_1^+ \lambda_2^- z}} \bigg). 
\end{align}

\end{corollary}

\begin{proposition}\label{prop2.50} Let $m,n>-1/2$, $0\leq|\beta_i|<\alpha_i$, $i=1,2$.
\vspace{2mm}

\noindent{1.} Suppose $m,n>0$. Then, as $z\rightarrow0$,
\begin{equation*}f_Z(z) \sim - \frac{\gamma_1^{2m+1} \gamma_2^{2n+1}   }{2 \pi \alpha_1^{2m} \alpha_2^{2n}} \frac{\Gamma(m)\Gamma(n)}{\Gamma(m+1/2) \Gamma(n+1/2) }\ln |z|.
\end{equation*}
\noindent{2.} Suppose $m=0$ and $n>0$. Then, as $z\rightarrow0$,
\begin{equation}\label{form6}f_Z(z)\sim \frac{\gamma_1 \gamma_2^{2n+1}   }{2 \pi^{3/2}  \alpha_2^{2n}    }  \frac{\Gamma(n)}{\Gamma(n+1/2)}(\ln |z|)^2 .
\end{equation}

\noindent{3.} Suppose $m=n=0$. Then, as $z\rightarrow0$, 
\begin{equation}\label{form7}f_Z(z)\sim  -\frac{\gamma_1 \gamma_2 }{3 \pi^2  } (\ln |z|)^3.
\end{equation}

\noindent{4.} Suppose $m<0$ and $m < n$. Then, as $z\rightarrow0$, 
\begin{equation}\label{form8}
f_Z(z) \sim  \frac{\gamma_1^{2m+1} \gamma_2^{2n+1}  }{2^{4m+2} \pi   } \frac{(\Gamma(-m))^2 \Gamma(n-m)}{\Gamma(m+1/2) \Gamma(n+1/2)}  |z|^{2m}.
\end{equation}

\noindent{5.} Suppose $m=n<0$. Then, as $z\rightarrow0$,
\begin{equation}\label{form9}
f_Z(z) \sim -\frac{(\gamma_1\gamma_2)^{2m+1}}{2^{4m+1} \pi   }  \frac{(\Gamma(-m))^2}{(\Gamma(m+1/2))^2} |z|^{2m}\ln|z|.
\end{equation}
In particular, the distribution of $Z$ is unimodal with mode 0 for all parameter constellations.
\end{proposition}

\begin{remark}  The limiting forms for the CDF and PDF of $Z$ given in Proposition \ref{cor2.7} and Corollary \ref{prop2.5} simplify for certain parameter values. For example, if $\lambda_1^-\lambda_2^->\lambda_1^+\lambda_2^+$, then $\mathrm{e}^{-2\sqrt{\lambda_1^-\lambda_2^-z}}\ll\mathrm{e}^{-2\sqrt{\lambda_1^+\lambda_2^+z}}$ as $z\rightarrow\infty$, and therefore the limiting form (\ref{limf1}) simplifies to
\begin{align}
\bar{F}_{Z}(z) & \sim  \frac{\sqrt{\pi} \gamma_1^{2m+1}\gamma_2^{2n+1}(\lambda_1^+)^{\frac{1}{4}(2n-2m-3)} (\lambda_2^+)^{\frac{1}{4}(2m-2n-3)} }{(2\alpha_1)^{m+1/2}(2\alpha_2)^{n+1/2}\Gamma(m+1/2)\Gamma(n+1/2)} z^{\frac{1}{4}(2m+2n-1)}   \mathrm{e}^{-2\sqrt{\lambda_1^+ \lambda_2^+ z}}, \quad z\rightarrow\infty. \nonumber
\end{align}
\end{remark}

\begin{remark}
The $\mathrm{VG}(m,\alpha,\beta,0)$ distribution is also unimodal; however, the mode is at the origin if and only if $m\leq1/2$ (see \cite[Section 2.8]{vg review}). We see that the growth of the singularity of the PDF $f_Z(z)$ in the limit $z\rightarrow0$ increases as $m$ and $n$ decrease below 0, which is to be expected given that $\mathrm{VG}(m,\alpha,\beta,0)$ PDF is bounded for $m>0$, has a logarithmic singularity at the origin for $m=0$, and a power law singularity at the origin for $-1/2<m<0$ \cite[Section 2.1]{vg review}. 
\end{remark}

For $0<p<1$, let $Q(p)=F^{-1}(p)$ denote the quantile function of the product $Z$. In the following proposition, we give asymptotic approximations for the quantile function.

\begin{proposition}\label{thm2}Let $m,n>-1/2$, $0\leq|\beta_i|<\alpha_i$, $i=1,2$. Also, let $\xi_1=\min\{\lambda_1^-\lambda_2^-,\lambda_1^+\lambda_2^+\}$ and $\xi_2=\min\{\lambda_1^-\lambda_2^+,\lambda_1^+\lambda_2^-\}$.
Then
\begin{align}
\label{q1}Q(p)&\sim\frac{1}{4\xi_1}(\ln(1-p))^2, \quad p\rightarrow1, \\
\label{q2}Q(p)&\sim-\frac{1}{4\xi_2}(\ln(p))^2, \quad p\rightarrow0.
\end{align}   
\end{proposition}



 
We now prove Proposition \ref{cor2.7}. We will make use of Lemma 2.1 of \cite{ad11}:

\begin{lemma}[Arendarczyk and D\c{e}bicki \cite{ad11}]\label{lem21} Let $X_1$ and $X_2$ be independent, non-negative random variables such that, as $x\rightarrow\infty$,
\[\bar{F}_{X_i}(x)\sim A_ix^{r_i}\exp(-b_ix^{a_i}), \quad i=1,2.\]
Then, as $x\rightarrow\infty$,
\[\bar{F}_{X_1X_2}(x)\sim  Ax^{r}\exp(-bx^{a}),\]
where
\begin{align*}
a&=\frac{a_1a_2}{a_1+a_2},\quad b=b_1^{\frac{a_2}{a_1+a_2}}b_2^{\frac{a_1}{a_1+a_2}}\bigg(\Big(\frac{a_1}{a_2}\Big)^{\frac{a_2}{a_1+a_2}}+\Big(\frac{a_2}{a_1}\Big)^{\frac{a_1}{a_1+a_2}}\bigg),   \\
r&=\frac{a_1a_2+2a_1r_2+2a_2r_1}{2(a_1+a_2)},\quad  A=\sqrt{2\pi}\frac{A_1A_2}{\sqrt{a_1+a_2}}(a_1b_1)^{\frac{a_2-2r_1+2r_2}{2(a_1+a_2)}}(a_2b_2)^{\frac{a_1-2r_2+2r_1}{2(a_1+a_2)}}.
\end{align*}
\end{lemma}

\noindent{\emph{Proof of Proposition \ref{cor2.7}.}}   We will denote the positive and negative parts of a random variable $X$ by $X^+$ and $X^-$, respectively. Let $X \sim \mathrm{VG}(m,\alpha_1,\beta_1,0)$ and $ Y \sim \mathrm{VG}(n,\alpha_2,\beta_2,0)$ be independent. The asymptotic behaviour of the CDFs of $X^+,X^-$ and $Y^+,Y^-$ are given in \cite{vg review} as follows: as $x,y \rightarrow \infty$,
\begin{align*}
\bar{F}_{X^+}(x) & \sim \bar{F}_X (x)  \sim  \frac{\gamma_1^{2m+1}}{(2\alpha_1)^{m+1/2}(\alpha_1-\beta_1)\Gamma(m+1/2)} x^{m-1/2} \mathrm{e}^{-(\alpha_1-\beta_1)x}, 
\\ \bar{F}_{X^-}(x) & \sim F_{X}(-x) \sim \frac{\gamma_1^{2m+1}}{(2\alpha_1)^{m+1/2} (\alpha_1+\beta_1) \Gamma(m+1/2) } x^{m-1/2} \mathrm{e}^{-(\alpha_1+\beta_1)x}, 
\\ \bar{F}_{Y^+}(y) & \sim \bar{F}_Y (y) \sim \frac{\gamma_2^{2n+1}}{(2\alpha_2)^{n+1/2}(\alpha_2-\beta_2)\Gamma(n+1/2)} y^{n-1/2} \mathrm{e}^{-(\alpha_2-\beta_2)y}, 
\\ \bar{F}_{Y^-} (y) & \sim F_Y(-y) \sim \frac{\gamma_2^{2n+1}}{(2\alpha_2)^{n+1/2} (\alpha_2+\beta_2) \Gamma(n+1/2) } y^{n-1/2} \mathrm{e}^{-(\alpha_2+\beta_2)y}. 
\end{align*}
 Applying Lemma \ref{lem21} now gives that, as $z\rightarrow\infty$,
\begin{align*}
\bar{F}_{X^+Y^+}(z) & \sim  \frac{\sqrt{\pi} \gamma_1^{2m+1}\gamma_2^{2n+1}}{(2\alpha_1)^{m+1/2}(2\alpha_2)^{n+1/2}\Gamma(m+1/2)\Gamma(n+1/2)} (\lambda_1^-)^{\frac{1}{4}(2n-2m-3)} (\lambda_2^-)^{\frac{1}{4}(2m-2n-3)} 
\\ & \qquad \times z^{\frac{1}{4}(2m+2n-1)} \mathrm{e}^{-2\sqrt{\lambda_1^-\lambda_2^-z}},
\\ \bar{F}_{X^-Y^-}(z) & \sim  \frac{ \sqrt{\pi} \gamma_1^{2m+1}\gamma_2^{2n+1}}{(2\alpha_1)^{m+1/2}(2\alpha_2)^{n+1/2}\Gamma(m+1/2)\Gamma(n+1/2)} (\lambda_1^+)^{\frac{1}{4}(2n-2m-3)} (\lambda_2^+)^{\frac{1}{4}(2m-2n-3)} 
\\ & \qquad \times z^{\frac{1}{4}(2m+2n-1)} \mathrm{e}^{-2\sqrt{\lambda_1^+\lambda_2^+z}}.
\end{align*}
The limiting form (\ref{limf1}) now follows from the formula
$\bar{F}_{XY} (z)  =  \bar{F}_{X^+Y^+}(z)+ \bar{F}_{X^-Y^-}(z)$.
The limiting form (\ref{limf2}) is  obtained similarly, since $F_{XY}(z)=F_{X^+Y^-}(z)+F_{X^-Y^+}(z)$. \hfill $\Box$

\vspace{3mm}
\noindent{\emph{Proof of Corollary \ref{prop2.5}.}} Since $f_Z(z)=-\bar{F}_Z'(z)$ and $f_Z(z)=F_Z'(z)$, we deduce the limiting forms (\ref{finfty1}) and (\ref{finfty2}) from differentiating the limiting forms (\ref{limf1}) and (\ref{limf2}), respectively, and retaining the highest order terms.
\hfill $\Box$

\vspace{3mm}

\noindent{\emph{Proof of Proposition \ref{prop2.50}.}} 
Recall that the PDF (\ref{eq:3}) is a double infinite series of Meijer $G$-functions. To ease notation, we will let $x=\alpha_1\alpha_2 z/4$. To determine the behaviour of $f_Z(z)$ as $z\rightarrow0$ we study the asymptotic behaviour as $x\downarrow0$ of the functions
\begin{align*}
 g_{i,j,m,n}(x) =G^{4,0}_{0,4} \bigg( x^2 \bigg| { - \atop \frac{j}{2}, i-\frac{j}{2},m+\frac{j}{2},n+i-\frac{j}{2}  }   \bigg),
\end{align*}
where $i,j\geq0$ are non-negative integers such that $j\leq 2i$, and it suffices to study the limit $x\downarrow0$ since $ g_{i,j,m,n}(x)$ is an even function in $x$. Using the contour integral definition of the Meijer $G$-function (\ref{mdef}) followed by an application of the residue theorem we can write
\begin{align}
&g_{i,j,m,n}(x)\nonumber  \\
&\quad= \frac{1}{2\pi \mathrm{i}}\int_{L} \Gamma\Big(s+\frac{j}{2}\Big)\Gamma\Big(s+i-\frac{j}{2}\Big) \Gamma\Big(s+m+\frac{j}{2}\Big) \Gamma\Big(s+n+i-\frac{j}{2}\Big)  x^{-2s} \,\mathrm{d}s\nonumber\\
\label{res1} &\quad= \sum_{s^\star \in A} \mathrm{Res}\bigg[\Gamma\Big(s+\frac{j}{2}\Big)\Gamma\Big(s+i-\frac{j}{2}\Big) \Gamma\Big(s+m+\frac{j}{2}\Big) \Gamma\Big(s+n+i-\frac{j}{2}\Big) x^{-2s} ,\:s=s^\star \bigg], 
\end{align}
where the integration path $L$ is a loop that encircles the poles of the gamma functions and $A$ is the collection of these poles. We now compute the residues in (\ref{res1}) for the five different cases of $m$ and $n$ as given in the statement of the proposition.

\vspace{2mm}

\noindent 1. Let $m,n>0$. Suppose first that $i=j=0$. We need to compute
\begin{equation}\label{res123}\mathrm{Res}\big[(\Gamma(s))^2 \Gamma(m+s) \Gamma(n+s)x^{-2s},\:s=s^*\big],
\end{equation}
for all poles $s^*$ of the gamma functions $\Gamma(s)$, $\Gamma(m+s)$ and $\Gamma(n+s)$. First, we consider the pole at $s=0$, for which is suffices to find the coefficient of $s^{-1}$ in the Laurent expansion of
\begin{equation}\label{eq:20}
(\Gamma(s))^2 \Gamma(m+s) \Gamma(n+s)x^{-2s}.
\end{equation}
We will make use of the Taylor expansion of $x^{-2s}$ about $s=0$, as given by
\begin{align}\label{exp8}
x^{-2s} =\exp(-2s\ln(x)) = \sum^\infty_{k=0} \frac{(-2 \ln (x))^k }{k!} s^k, 
\end{align}
as well as the standard asymptotic expansion
\begin{equation}\label{exp9}
\Gamma(s)=\frac{1}{s}+\gamma +o(1), \quad s\rightarrow0 ,   
\end{equation}
where $\gamma$ is the Euler-Mascheroni constant. Using (\ref{exp8}) and (\ref{exp9}) we find after a simple calculation that the coefficient of $s^{-1}$ in the Laurent expansion of (\ref{eq:20}) is given by $\Gamma(m)\Gamma(n)[-2\ln(x)+2\gamma]$, and therefore
\begin{align}
\mathrm{Res}\big[(\Gamma(s))^2 \Gamma(m+s) \Gamma(n+s)x^{-2s},\:s=0\big]&=-2\Gamma(m)\Gamma(n)[\ln(x)-\gamma]\nonumber\\
\label{res456}&\sim -2\Gamma(m)\Gamma(n)\ln(x), \quad x\downarrow0.
\end{align}
It is easily seen that for all other poles besides the one at $s=0$ the residue (\ref{res123}) is of a smaller asymptotic order in the limit $x\downarrow0$ than the residue (\ref{res456}). Similarly, it is readily seen that for $i\geq1$ or $j\geq1$ the residues in (\ref{res1}) are of a smaller asymptotic order in the limit $x\downarrow0$ than the residue (\ref{res456}). Finally, we conclude that, as $x\rightarrow0$,
\begin{equation*}
g_{0,0,m,n}(x) \sim -2\Gamma(m)\Gamma(n)\ln|x|\sim -2\Gamma(m)\Gamma(n)\ln|z|,
\end{equation*}
(where we can state the limiting form for $x\rightarrow\infty$ as it sufficed to consider the limit $x\downarrow0$) and that $g_{i,j,m,n}(x)$ is of smaller asymptotic order in the limit $x\rightarrow0$ (equivalently in the limit $z\rightarrow0$) for $i\geq1$ or $j\geq1$. Thus, as $z\rightarrow0$,
\begin{align*}
f_Z(z)&\sim \frac{\gamma_1^{2m+1} \gamma_2^{2n+1} }{4 \pi  \Gamma(m+1/2) \Gamma(n+1/2) \alpha_1^{2m} \alpha_2^{2n} }\cdot g_{0,0,m,n}(x) \\
&\sim  - \frac{\gamma_1^{2m+1} \gamma_2^{2n+1}   }{2 \pi \alpha_1^{2m} \alpha_2^{2n}} \frac{\Gamma(m)\Gamma(n)}{\Gamma(m+1/2) \Gamma(n+1/2) }\ln |z|.
\end{align*}


\vspace{2mm}

\noindent 2. Let $m=0$, $n>0$. Arguing similarly to in part 1, we find that the dominant residue in (\ref{res1}) in the limit $x\downarrow0$ is given by
\begin{align*}
\mathrm{Res}\big[(\Gamma(s))^3 \Gamma(n+s)x^{-2s},\:s=0\big]\sim 2\Gamma(n)(\ln(x))^2, \quad x\downarrow0,
\end{align*}
with all other residues being of lower asymptotic order in the limit $x\downarrow0$. Arguing as in part 1, we deduce the limiting form (\ref{form6}).

\vspace{2mm}

\noindent 3. Let $m=n=0$. This time the dominant residue in (\ref{res1}) in the limit $x\downarrow0$ is given by
\begin{align*}
\mathrm{Res}\big[(\Gamma(s))^4x^{-2s},\:s=0\big]\sim-\frac{4}{3}(\ln(x))^3, \quad x\downarrow0,
\end{align*}
and we thus obtain the limiting form (\ref{form7}).

\vspace{2mm}

\noindent 4. Let $m<0$ and $m<n$. Since $m<0$, the dominant residue now is located at the pole $s=-m$, and is given by
\begin{align*}
&\mathrm{Res}\big[(\Gamma(s))^2\Gamma(m+s)\Gamma(n+s)x^{-2s},\:s=-m\big]\\
&\quad=x^{2m}\mathrm{Res}\big[(\Gamma(t-m))^2\Gamma(t)\Gamma(n-m+t)x^{-2t},\:t=0\big]\\
&\quad\sim (\Gamma(-m))^2 \Gamma(n-m) x^{2m}, \quad x\downarrow0,
\end{align*}
and we thus obtain the limiting form (\ref{form8}).

\vspace{2mm}

\noindent 5. Let $m=n<0$. As $m=n<0$, the dominant pole is again located at the pole $s=-m$, and is this time given by
\begin{align*}
&\mathrm{Res}\big[(\Gamma(s))^2\Gamma(m+s)\Gamma(n+s)x^{-2s},\:s=-m\big]\\
&\quad=x^{2m}\mathrm{Res}\big[(\Gamma(t-m))^2(\Gamma(t))^2x^{-2t},\:t=0\big]\\
&\quad\sim -2(\Gamma(-m))^2  x^{2m}\ln(x), \quad x\downarrow0.
\end{align*}
We thus obtain the limiting form (\ref{form9}).

\vspace{2mm}

The final assertion that the distribution of $Z$ is unimodal with mode 0 for all parameter constellations follows because parts 1--5 imply that the PDF $f_Z(z)$ has a singularity at the origin for all parameter values, and the PDF is bounded everywhere except for the singularity at $z=0$.
 \hfill $\Box$

\vspace{3mm}

To prove Proposition \ref{thm2} we will make use of the following lemma.
\begin{lemma}\label{lemq}
Let $a,A,z>0$ and $r\in\mathbb{R}$. Let $g:(0,\infty)\rightarrow\mathbb{R}$ be a function such that $g(x)\rightarrow0$ as $x\rightarrow\infty$.  Consider the solution $x$ of the equation
\begin{equation}\label{zeqn}Ax^{r}\exp(-a\sqrt{x})\big(1+g(x)\big)=z,
\end{equation}
and observe that there is a unique solution for sufficiently small $z$. Then, as $z\rightarrow0$,
\begin{align}\label{lemqeqn}x\sim \frac{1}{a^2}(\ln(z))^2.
\end{align}
\end{lemma}

\begin{proof} For ease of exposition, we set $w=A/z$ and $h(x)=1+g(x)$. Rearranging equation (\ref{zeqn}) gives that
\begin{align}
\label{ax2}x=\frac{1}{a^2}\big(\ln(w)+r\ln(x)+\ln(h(x))\big)^2. 
\end{align}
It is clear from equation (\ref{ax2}) that $x=O((\ln(w))^2)$ as $w\rightarrow\infty$, from which we deduce that $\ln(x)=O(\ln(\ln(w)))$ as $w\rightarrow\infty$, and that $\ln(h(x))\rightarrow0$ as $w\rightarrow\infty$ (since $\ln(1+u)\rightarrow0$ as $u\rightarrow0$). We therefore deduce that $x\sim (\ln(w))^2/a^2\sim (\ln(z))^2/a^2$, as $z\rightarrow0$, as required.
\end{proof}   

\noindent{\emph{Proof of Proposition \ref{thm2}.}} First, we derive the limiting form (\ref{q1}). The quantile function $Q(p)$ solves the equation $\bar{F}_Z(Q(p))=1-p$. On applying the limiting form (\ref{limf1}) we see that $Q(p)$ solves an equation of the form (\ref{zeqn}) from Lemma \ref{lemq} with
$z=1-p$ and $a=2\min\{(\lambda_1^-\lambda_2^-)^{1/2},(\lambda_1^+\lambda_2^+)^{1/2}\}$.  Applying the limiting form (\ref{lemqeqn}) with these values of $z$ and $a$ yields the limiting form (\ref{q1}). 

We derive the limiting form (\ref{q2}) similarly, this time using the fact that $Q(p)$ solves the equation $F_Z(Q(p))=p$ and applying the limiting form (\ref{limf2}).
\hfill $\Box$

\section{Special cases}\label{sec3}

In this section, we provide simplified formulas for the PDF, CDF and characteristic function of the product $Z=XY$ for some important special cases. The formulas are expressed in terms of the sine and cosine integrals and Bessel and Struve functions, all of which are defined in Appendix \ref{appa}.

\subsection{Product of independent asymmetric Laplace random variables}\label{sec3.1}


As noted by \cite{kkp01}, the $\mathrm{VG}(1/2,\alpha,\beta,0)$ distribution (where $0\leq|\beta|<\alpha$) corresponds to the asymmetric Laplace distribution (with zero location parameter) with PDF
\begin{equation}\label{alpdf}f(x)=\frac{\alpha^2-\beta^2}{2\alpha}\mathrm{e}^{\beta x-\alpha|x|}, \quad x\in\mathbb{R}.
\end{equation}
If a random variable $X$ has PDF (\ref{alpdf}), then we write $X\sim \mathrm{AL}(\alpha,\beta)$. Further setting $\beta=0$ yields the classical Laplace distribution with PDF
\begin{equation}\label{lpdf}f(x)=\frac{\alpha}{2}\mathrm{e}^{-\alpha|x|}, \quad x\in\mathbb{R}.
\end{equation}
If a random variable $X$ has PDF (\ref{lpdf}), then we write $X\sim \mathrm{Laplace}(\alpha)$.

\begin{corollary}\label{31}Let $X\sim\mathrm{AL}(\alpha_1,\beta_1)$ and $Y\sim\mathrm{AL}(\alpha_2,\beta_2)$ be independent random variables, where $0\leq |\beta_i|<\alpha_i$, $i=1,2$. Denote their product by $Z=XY$. Then

\vspace{2mm}

\noindent (i) For $z\in\mathbb{R}$,
\begin{equation}\label{lappdf1}f_Z(z)=\frac{\gamma_1^2\gamma_2^2}{2\alpha_1\alpha_2}\Big\{K_0\big(2\sqrt{\alpha_1-\beta_1}\sqrt{\alpha_2|z|-\beta_2z}\big)+K_0\big(2\sqrt{\alpha_1+\beta_1}\sqrt{\alpha_2|z|+\beta_2z}\big)\Big\}.
\end{equation}

\noindent (ii) For $z>0$,
\begin{align*}
F_Z(z)=1-\frac{\gamma_1^2\gamma_2^2}{2\alpha_1\alpha_2}\bigg\{\frac{\sqrt{z}}{(\lambda_1^-\lambda_2^-)^{1/2}}K_1\big(2(\lambda_1^-\lambda_2^-)^{1/2}\sqrt{z}\big)+\frac{\sqrt{z}}{(\lambda_1^+\lambda_2^+)^{1/2}}K_1\big(2(\lambda_1^+\lambda_2^+)^{1/2}\sqrt{z}\big)\bigg\},    
\end{align*}
and, for $z<0$,
\begin{align*}
F_Z(z)=\frac{\gamma_1^2\gamma_2^2}{2\alpha_1\alpha_2}\bigg\{\frac{\sqrt{-z}}{(\lambda_1^-\lambda_2^+)^{1/2}}K_1\big(2(\lambda_1^-\lambda_2^+)^{1/2}\sqrt{-z}\big)+\frac{\sqrt{-z}}{(\lambda_1^+\lambda_2^-)^{1/2}}K_1\big(2(\lambda_1^+\lambda_2^-)^{1/2}\sqrt{-z}\big)\bigg\}.
\end{align*}
\noindent (iii) For $t\in\mathbb{R}$,
\begin{align}\varphi_Z(t)&=\frac{\gamma_1^2\gamma_2^2}{4\alpha_1\alpha_2}\bigg\{\frac{(-\mathrm{i}t)^{-1/2}}{(\lambda_1^-\lambda_2^-)^{1/2}}\exp\bigg(\frac{\mathrm{i}\lambda_1^-\lambda_2^-}{2t}\bigg)W_{-\frac{1}{2},0}\bigg(\frac{\mathrm{i}\lambda_1^-\lambda_2^-}{2t}\bigg)\nonumber\\
&\quad+\frac{(-\mathrm{i}t)^{-1/2}}{(\lambda_1^+\lambda_2^+)^{1/2}}\exp\bigg(\frac{\mathrm{i}\lambda_1^+\lambda_2^+}{2t}\bigg)W_{-\frac{1}{2},0}\bigg(\frac{\mathrm{i}\lambda_1^+\lambda_2^+}{2t}\bigg)\nonumber\\
&\quad+\frac{(\mathrm{i}t)^{-1/2}}{(\lambda_1^-\lambda_2^+)^{1/2}}\exp\bigg(-\frac{\mathrm{i}\lambda_1^-\lambda_2^+}{2t}\bigg)W_{-\frac{1}{2},0}\bigg(-\frac{\mathrm{i}\lambda_1^-\lambda_2^+}{2t}\bigg)\nonumber\\
&\quad+\frac{(\mathrm{i}t)^{-1/2}}{(\lambda_1^+\lambda_2^-)^{1/2}}\exp\bigg(-\frac{\mathrm{i}\lambda_1^+\lambda_2^-}{2t}\bigg)W_{-\frac{1}{2},0}\bigg(-\frac{\mathrm{i}\lambda_1^+\lambda_2^-}{2t}\bigg)\bigg\}. \nonumber
\end{align}
\end{corollary}

\begin{corollary}\label{32} Let $X\sim\mathrm{Laplace}(\alpha_1)$ and $Y\sim\mathrm{Laplace}(\alpha_2)$ be independent random variables, and denote their product by $Z=XY$. Then

\vspace{2mm}

\noindent (i) For $z\in\mathbb{R}$,
\begin{align}\label{lappdf2}
f_Z(z)  = \alpha_1 \alpha_2 K_{0}\big(2\sqrt{\alpha_1\alpha_2 |z| }\big), 
\end{align}

\noindent (ii) For $z>0$,
\begin{align}\label{lappdf3}
F_Z(z)=1-\sqrt{\alpha_1\alpha_2z}K_1\big(2\sqrt{\alpha_1\alpha_2z}\big),
\end{align}
and, for $z<0$,
\begin{align*}
F_Z(z)=\sqrt{\alpha_1\alpha_2(-z)}K_1\big(2\sqrt{\alpha_1\alpha_2(-z)}\big).
\end{align*}
Moreover, the following formula is valid for all $z\in\mathbb{R}$:
\begin{align}
\label{lefg}F_Z(z)=\frac{1}{2}+\mathrm{sgn}(z)\bigg[\frac{1}{2}-\sqrt{\alpha_1\alpha_2|z|}K_1\big(2\sqrt{\alpha_1\alpha_2|z|}\big)\bigg].
\end{align}

\noindent (iii) For $t\in\mathbb{R}$,
\begin{align}\varphi_Z(t)=\frac{\alpha_1\alpha_2}{|t|}\bigg\{\bigg(\frac{\pi}{2}-\,\mathrm{Si}\Big(\frac{\alpha_1\alpha_2}{|t|}\Big)\bigg)\cos\Big(\frac{\alpha_1\alpha_2}{|t|}\Big)+\,\mathrm{Ci}\Big(\frac{\alpha_1\alpha_2}{|t|}\Big)\sin\Big(\frac{\alpha_1\alpha_2}{|t|}\Big)\bigg\}. \label{2345}
\end{align}
\end{corollary}

\begin{remark} Formulas (\ref{lappdf2}) and (\ref{lappdf3}) for the PDF and CDF of the product of two independent symmetric Laplace random variables are in agreement with the formulas of \cite{nad07} for the PDF and CDF of the product $|XY|$. However, to the best of our knowledge, the other formulas of Corollaries \ref{31} and \ref{32} are new.   
\end{remark}

\noindent{\emph{Proof of Corollary \ref{31}.}} (i): Set $m=n=1/2$ in the formula of Proposition \ref{cor.2}.

\vspace{2mm}

\noindent (ii): The proof is similar to that of Proposition \ref{pmn}, but we calculate the integral using (\ref{kint3}) instead of (\ref{kint2}).

\vspace{2mm}

\noindent (iii): Set $m=n=1/2$ in (\ref{pok}).
\hfill $\Box$

\vspace{3mm}

\noindent{\emph{Proof of Corollary \ref{32}.}} (i) and (ii): Set $\beta_1=\beta_2=0$ in parts (i) and (ii) of Corollary \ref{31}.

\vspace{2mm}

\noindent (iii): Since the PDF (\ref{lappdf2}) is symmetric about the origin, we evaluate $\varphi_Z(t)=\mathbb{E}[\cos(tZ)]=2\int_{0}^\infty \cos(tz)f_Z(z)\,\mathrm{d}z$ using the integral formula (\ref{thn0}).
\hfill $\Box$


\subsection{Mixed product of independent zero mean normal and  Laplace random variables}\label{sec3.2}

Let $N_1\sim N(0,\sigma_1^2)$ and $N_2\sim N(0,\sigma_2^2)$ be independent normal random variables. Then $N_1N_2\sim \mathrm{VG}(0,1/(\sigma_1\sigma_2),0,0)$ (see 
\cite[Proposition 1.2]{gaunt vg}).
Recall also 
that the $\mathrm{VG}(1/2,\alpha,0,0)$ distribution corresponds to a Laplace distribution with PDF $f(x)=(\alpha/2)\mathrm{e}^{-\alpha|x|}$, $x\in\mathbb{R}$. In the following corollary, we provide formulas for the PDF, CDF and characteristic function of the product $Z=N_1N_2Y$, where $N_1\sim N(0,\sigma_1^2)$, $N_2\sim N(0,\sigma_2^2)$ and $Y\sim\mathrm{Laplace}(\alpha_2)$ are mutually independent. 

\begin{corollary}\label{corrrr} Suppose that $N_1\sim N(0,\sigma_1^2)$, $N_2\sim N(0,\sigma_2^2)$ and $Y\sim\mathrm{Laplace}(\alpha_2)$ are mutually independent random variables. Denote their product by $Z=N_1N_2Y$. Let $\alpha_1=1/(\sigma_1\sigma_2)$. Then

\vspace{2mm}

\noindent (i) For $z\in\mathbb{R}$,
\begin{align}\label{corrrr8}
f_Z(z)=\frac{2\alpha_1\alpha_2}{\pi}K_0\big((-1)^{1/4}\sqrt{2\alpha_1\alpha_2|z|}\big)K_0\big((-1)^{-1/4}\sqrt{2\alpha_1\alpha_2|z|}\big).
\end{align}  

\noindent (ii) For $z\in\mathbb{R}$,
\[F_Z(z)= \frac{1}{2} + \frac{\alpha_1 \alpha_2 z}{8 \pi^{3/2} } G^{4,1}_{1,5}\bigg( \frac{\alpha_1^2 \alpha_2^2}{16}z^2 \bigg|\,{ \frac{1}{2} \atop 0,0,0,\frac{1}{2},-\frac{1}{2}}\bigg).\]

\noindent (iii) For $t\in\mathbb{R}$,   
\begin{align}\label{bbb}
\varphi_Z(t)= \frac{\pi\alpha_1\alpha_2}{2|t|}\mathbf{K}_0\bigg(\frac{\alpha_1\alpha_2}{|t|}\bigg),  
\end{align}
where the Struve function $\mathbf{K}_0$ is defined as in (\ref{boldk}).
\end{corollary}

\begin{proof} (i) Apply the reduction formula (\ref{redm}) to the PDF (\ref{bb00}) (with $m=0$, $n=1/2$ and $\beta_1=\beta_2=0$).

\vspace{2mm}

\noindent (ii) Apply (\ref{fgh}) with $m=0$, $n=1/2$.

\vspace{2mm}

\noindent (iii) Apply the reduction formula (\ref{g30}) (with $a=b=0$) to (\ref{char}) (with $m=0$, $n=1/2$). 
\end{proof}

We now consider the more general setting that $X\sim\mathrm{VG}(m,\alpha_1,0,0)$ and $Y\sim\mathrm{Laplace}(\alpha_2)$ are independent random variables. In this case, a  simple formula can still be given for the characteristic function of the product $Z=XY$. The formula (\ref{bbbb}) is valid provided $m+1/2\in\mathbb{R}^+\setminus\mathbb{Z}$; if this condition is not met, then one can apply Proposition \ref{fghj} instead.

\begin{corollary} Let $X\sim\mathrm{VG}(m,\alpha_1,0,0)$ and $Y\sim\mathrm{Laplace}(\alpha_2)$ be independent random variables, and denote their product by $Z=XY$. 
Suppose that
 $m+1/2\in\mathbb{R}^+\setminus\mathbb{Z}$. Then, for $t\in\mathbb{R}$,   
\begin{align}\label{bbbb}
\varphi_Z(t)= \frac{\pi^{3/2}}{\cos(\pi m)\Gamma(m+1/2)}\bigg(\frac{\alpha_1\alpha_2}{2|t|}\bigg)^{m+1}\mathbf{K}_{-m}\bigg(\frac{\alpha_1\alpha_2}{|t|}\bigg),
\end{align}
where the Struve function $\mathbf{K}_{-m}$ is defined as in (\ref{boldk}).
\end{corollary}

\begin{proof}Apply the reduction formula (\ref{g301}) (with $a=1/2$, $b=m$) to (\ref{char}) (with $n=1/2$). 
\end{proof}

\subsection{Product of correlated zero mean normal random variables}\label{sec3.3}

Let $(U_i,V_i)^{\intercal}$, $i=1,2$, be independent bivariate normal random vectors with  zero mean vector, variances $(\sigma_{U_i}^2,\sigma_{V_i}^2)^{\intercal}$ and correlation coefficient $\rho_i$. It was noted by \cite{gaunt vg,gaunt prod} that the product $W_i=U_iV_i$ has a VG distribution,
\begin{equation}\label{vgrep}W_i\sim\mathrm{VG}\bigg(0,\frac{1}{\sigma_{U_i}\sigma_{V_i}(1-\rho_i^2)},\frac{\rho_i}{\sigma_{U_i}\sigma_{V_i}(1-\rho_i^2)},0\bigg),
\end{equation}
 which yielded a new derivation of the PDF of the product $W_i$ which was obtained independently by \cite{gr,np16,wb32}.  In the following corollary, we provide formulas for the PDF, CDF and characteristic function of the product $Z=W_1W_2$. 

\begin{corollary}\label{cor2.10} Let the previous notations prevail. Also, let $s=\sigma_{U_1}\sigma_{V_1}\sigma_{U_2}\sigma_{V_2}$.

\vspace{2mm}

\noindent (i) For $z\in\mathbb{R}$,
\begin{align}\label{log}
f_Z(z) & =  \frac{1}{4 \pi^2 s \sqrt{1-\rho_1^2} \sqrt{1-\rho_2^2} } \sum^\infty_{i=0} \sum^{2i}_{j=0}(\mathrm{sgn}(z))^j  \frac{(2 \rho_1)^j}{j!} \frac{(2 \rho_2)^{2i-j}}{(2i-j)!}  \nonumber\\
&\qquad\times G^{4,0}_{0,4}\bigg( \frac{z^2}{16s^2(1-\rho_1^2)^2 (1-\rho_2^2)^2 } \bigg|{ - \atop \frac{j}{2},\frac{j}{2},i-\frac{j}{2},i-\frac{j}{2} }\bigg) .
\end{align}
\noindent (ii) Suppose now that $\rho_1=\rho_2=0$. Then, for $z\in\mathbb{R}$,
\begin{equation*}F_Z(z)= \frac{1}{2} + \frac{ z}{8 \pi^2 s} G^{4,1}_{1,5}\bigg( \frac{z^2}{16s^2}  \bigg|\,{ \frac{1}{2} \atop 0,0,0,0,-\frac{1}{2}}\bigg).
\end{equation*}
\noindent (iii) Suppose that $\rho_1=\rho_2=0$. Then, for $t\in\mathbb{R}$,
\begin{align}
\label{aaaa}
\varphi_Z(t)=\frac{\pi}{4s|t|}\bigg\{J_0^2\bigg(\frac{1}{2s|t|}\bigg)+Y_0^2\bigg(\frac{1}{2s|t|}\bigg)\bigg\} . 
\end{align}
\end{corollary}

\begin{proof}
(i)  Combine (\ref{eq:3}) with (\ref{vgrep}).

\vspace{2mm}

\noindent (ii) Combine (\ref{fgh}) with (\ref{vgrep}) (with $\rho_1=\rho_2=0$).

\vspace{2mm}

\noindent (iii) We obtain (\ref{aaaa}) by applying the reduction formula (\ref{g300}) (with $a=b=0$) to (\ref{char}). 
\end{proof}

\begin{remark}
The formula (\ref{aaaa}) can be connected to the modified Bessel function $K_0$ by Nicholson's formula (\ref{nich}):
\begin{equation*}
\varphi_Z(t)=\frac{2}{\pi s|t|} \int_0^\infty K_0\bigg(\frac{\sinh x}{s|t|}\bigg) \,\mathrm{d}x. 
\end{equation*}   
\end{remark}

\begin{remark}\label{remend} 1. Note that $Z=_d N_1N_2N_3N_4$, where $(N_1,N_2,N_3,N_4)^\intercal$ is a multivariate normal random vector with zero mean vector and covariance matrix
\begin{equation*}
\Sigma = \left[\begin{aligned} \sigma_{1}^2 & & \rho_1 \sigma_{1} \sigma_{2} & & 0 & & 0 
\\ \rho_1 \sigma_{1} \sigma_{2} & & \sigma_{2}^2 & & 0 & & 0
\\ 0 & & 0 & & \sigma_{3}^2 & & \rho_2 \sigma_{3} \sigma_{4}
\\ 0 & & 0 & & \rho_2 \sigma_{3} \sigma_{4} & & \sigma_{4}^2 \end{aligned} \right], 
\end{equation*}
for which $s=\sigma_1\sigma_2\sigma_3\sigma_4$. In the case $\rho_1=\rho_2=0$, we have that $Z$ is thus distributed as the product of four independent zero mean normal random variables, and setting $\rho_1=\rho_2=0$ in the PDF (\ref{log}) gives the formula
\begin{equation}\label{hlog29}f_Z(z)=\frac{1}{4 \pi^2 s  }   G^{4,0}_{0,4}\bigg( \frac{z^2}{16s^2} \bigg|{ - \atop 0,0,0,0 }\bigg),\quad z\in\mathbb{R},
\end{equation} 
which is a special case of a more general formula of \cite{product normal} for the product of $k\geq2$ independent zero mean normal random variables. 

Whilst formulas for the PDF for the product of two correlated zero mean normal random variables and the product of $k$ independent zero mean normal random variables are known in the literature, the case of the product of three or more correlated zero mean normal random variables is not well-understood. To the best of our knowledge, the formula (\ref{log}) is one of the first such explicit formulas, and we hope that the result inspires other researchers to find further explicit formulas for the PDF of the product of three of more correlated zero mean normal random variables. 

\vspace{2mm}

\noindent 2. We note that formulas, expressed in terms of the Meijer $G$-function, for the CDF of the product of $k$ independent standard normal random variables have previously been given by \cite{s17}.

\vspace{2mm}

\noindent 3. A formula, in terms of the Meijer $G$-function, for the characteristic function of $Z=N_1N_2N_3N_4$ (with $\rho_1=\rho_2=0$) has previously been given by \cite[Propostion 2.2]{g17}, as a special case of a general formula for the characteristic function of the product of $k$ independent zero mean normal random variables. Formula (\ref{aaaa}) provides a simpler expression for the characteristic function for the case of the product of $k=4$ independent zero mean normal random variables and complements other simpler formulas for the cases $k=2$ \cite[Example 11.22]{so87} and $k=3$ \cite[Propostion 2.2]{g17}.

The results can be summarised as follows. Let $X_1,\ldots,X_k$ be independent normal random variables with $X_i\sim N(0,\sigma_i^2)$, $i=1,\ldots,k$. Denote the characteristic function of the product $\prod_{i=1}^kX_i$ by $\varphi_k(t)$. Let $s=\sigma_1\cdots\sigma_k$. Then, for $k\geq2$,
\begin{equation*}\varphi_k(t)=\frac{1}{\pi^{(k-1)/2}}G_{1,k-1}^{k-1,1}\bigg(\frac{1}{2^{k-2}s^2t^2}\;\bigg|\;\begin{matrix} 1 \\
\frac{1}{2},\ldots,\frac{1}{2} \end{matrix} \bigg).
\end{equation*} 
In the cases $k=2,3,4$, this formula simplifies to
\begin{align*}\varphi_2(t)&=\frac{1}{\sqrt{1+s^2t^2}},\quad \varphi_3(t)=\frac{1}{\sqrt{2\pi s^2t^2}}\exp\bigg(\frac{1}{4s^2t^2}\bigg)K_0\bigg(\frac{1}{4s^2t^2}\bigg), \\
\varphi_4(t)&=\frac{\pi}{4s|t|}\bigg\{J_0^2\bigg(\frac{1}{2s|t|}\bigg)+Y_0^2\bigg(\frac{1}{2s|t|}\bigg)\bigg\}.
\end{align*}
We could not find reduction formulas in the literature to obtain further simplifications for $k\geq5$.
\end{remark}

\section{Reduction formulas for the Meijer $G$-function}\label{sec4}

As a by-product of our analysis, we obtain the following reduction formulas for the Meijer $G$-function, which we believe to be new.

\begin{corollary}\label{redc}Suppose that $a,b\geq0$ are non-negative integers and that $c\in\mathbb{R}$. Then, for $x>0$,
\begin{align}\label{close1}
&G^{4,0}_{0,4} \bigg( x \bigg| { - \atop c, c,c+a+\frac{1}{2},c+b+\frac{1}{2} } \bigg)\nonumber\\
&\quad=4\pi x^c\sum_{i=0}^a\sum_{j=0}^b\frac{1}{4^{i+j}}\frac{(a+i)!}{i!(a-i)!}\frac{(b+j)!}{j!(b-j)!}x^{\frac{1}{4}(a+b-i-j)}K_{a-b-i+j}(4x^{1/4}),
\end{align}
and
\begin{align}\label{close2}
&G^{4,1}_{1,5} \bigg( x \bigg| { c+1 \atop c+\frac{1}{2}, c+\frac{1}{2},c+a+1,c+b+1,c } \bigg)\nonumber \\
&\quad=\pi x^c\sum_{i=0}^a\sum_{j=0}^b\frac{1}{2^{a+b+i+j}}\frac{(a+i)!}{i!}\frac{(b+j)!}{j!}G_{a+b+1-i-j,a-b-i+j}(4x^{1/4}),
\end{align}
where the function $G_{\mu,\nu}(x)$ is defined as in (\ref{gmn2}). 
In the case $a=b=0$, the formula (\ref{close2}) simplifies to
\begin{equation}\label{close3} G^{4,1}_{1,5} \bigg( x \bigg| { c+1 \atop c+\frac{1}{2}, c+\frac{1}{2},c+1,c+1,c } \bigg) =\pi x^c\big(1-4x^{1/4}K_1(4x^{1/4})\big), \quad x>0.    
\end{equation}
\end{corollary}

\begin{proof}
To obtain (\ref{close1}) we combine equations (\ref{bb00}) and (\ref{eq:17}) (with $\beta_1=\beta_2=0$) and then apply (\ref{meijergidentity}).  To obtain (\ref{close2}) we combine equations (\ref{near}) and (\ref{73}) (with $\beta_1=\beta_2=0$) and then apply (\ref{meijergidentity}).  To obtain (\ref{close3}) we combine (\ref{near}) and (\ref{lefg}) and then apply (\ref{meijergidentity}).
\end{proof}

\begin{remark}
In the case $a=b=0$, the reduction formula (\ref{close1}) simplifies to
\begin{align*}
G^{4,0}_{0,4} \bigg( x \bigg| { - \atop c, c,c+\frac{1}{2},c+\frac{1}{2} } \bigg)=4\pi x^c K_{0}(4x^{1/4}),
\end{align*}
which is in agreement with formula (\ref{redm0}) when specialised to $a=b(=c)$.
\end{remark}

\appendix

\section{Special functions}\label{appa}
In this appendix, we define several special functions, and state some of their basic properties that are needed in this paper. Unless otherwise stated, these and further properties can be found in the standard references \cite{gradshetyn,luke,olver}. 

\subsection{Bessel, Struve and Lommel functions}

The \emph{modified Bessel function of the second kind} can be defined, for $\nu\in\mathbb{R}$ and $x>0$, by
\begin{equation}\label{kdefn}
K_{\nu} (x) = \frac{x^\nu}{2^{\nu+1}}  \int^\infty_0 t^{-(\nu+1)} \exp\bigg(-t-\frac{x^2}{4t} \bigg) \, \mathrm{d}t.
\end{equation}
For $m=0,1,2,\ldots$, we have the elementary representation
\begin{equation}\label{special} K_{m+1/2}(x)=\sqrt{\frac{\pi}{2x}}\sum_{j=0}^m\frac{(m+j)!}{(m-j)!j!}(2x)^{-j}\mathrm{e}^{-x}.
\end{equation}
For $a>0$, $r>|\nu|$, we have the following definite integral formula:
\begin{equation}\label{eq:32}\int_0^\infty t^{r-1}K_\nu(at)\,\mathrm{d}t=\frac{2^{r-2}}{a^r}\Gamma\Big(\frac{r-\nu}{2}\Big)\Gamma\Big(\frac{r+\nu}{2}\Big).
\end{equation}

The \emph{Bessel function of the first kind} is defined, for $\nu\in\mathbb{R}$ and $x>0$, by the power series
\begin{equation*}
J_\nu(x)=\sum_{k=0}^\infty (-1)^k\frac{(x/2)^{2k+\nu}}{k!\Gamma(k+\nu+1)}.    
\end{equation*}
The \emph{Bessel function of the second kind} can be defined, for $\nu\in\mathbb{R}$ and $x>0$, by
\begin{equation*}
Y_\nu(x)=   \frac{1}{\pi}\int_0^\pi\sin(x\sin t-\nu t)\,\mathrm{d}t-\frac{1}{\pi}\int_0^\infty\big(\mathrm{e}^{\nu t}+\mathrm{e}^{-\nu t}\cos(\nu\pi)\big)\mathrm{e}^{-x\sinh t}\,\mathrm{d}t, 
\end{equation*}
which is 
the so–called Gubler–Weber formula. Nicholson's formula states that
\begin{equation}
\label{nich} J_\nu^2(x)+Y_\nu^2(x)=\frac{8}{\pi^2}\int_0^\infty \cosh(2\nu t)K_0(2x\sinh t)\,\mathrm{d}t.   
\end{equation}
The \emph{Hankel functions of the first and second kind} are defined as
\begin{equation*}
H_\nu^{(1)}(x)=J_\nu(x)+\mathrm{i}Y_\nu(x), \quad H_\nu^{(2)}(x)=J_\nu(x)-\mathrm{i}Y_\nu(x).
\end{equation*}
The following basic identity holds
\begin{equation}\label{jhjh}H_\nu^{(1)}(x)H_\nu^{(2)}(x)=J_\nu^2(x)+Y_\nu^2(x).
\end{equation}

The \emph{Struve functions} are defined, for $\nu\in\mathbb{R}$ and $x>0$, by
\begin{equation*}
\mathbf{H}_\nu(x)=\sum_{k=0}^\infty (-1)^k\frac{(x/2)^{2k+\nu+1}}{\Gamma(k+3/2)\Gamma(k+\nu+3/2)} 
\end{equation*}
and
\begin{equation}
\mathbf{K}_\nu(x)=\mathbf{H}_\nu(x)-Y_\nu(x). \label{boldk}
\end{equation}

The \emph{modified Lommel function of the first kind} is defined by the hypergeometric series
\begin{align}t_{\mu,\nu}(x)&=\frac{x^{\mu+1}}{(\mu-\nu+1)(\mu+\nu+1)} {}_1F_2\bigg(1;\frac{\mu-\nu+3}{2},\frac{\mu+\nu+3}{2};\frac{x^2}{4}\bigg) \nonumber\\
&=2^{\mu-1}\Gamma\bigg(\frac{\mu-\nu+1}{2}\bigg)\Gamma\bigg(\frac{\mu+\nu+1}{2}\bigg)\sum_{k=0}^\infty\frac{(\frac{1}{2}x)^{\mu+2k+1}}{\Gamma\big(k+\frac{\mu-\nu+3}{2}\big)\Gamma\big(k+\frac{\mu+\nu+3}{2}\big)}. \nonumber
\end{align}
In this paper, we use the following normalisation of the modified Lommel function of the first kind which was introduced by \cite{gaunt lommel}:
\begin{align*}\tilde{t}_{\mu,\nu}(x)&=\frac{1}{2^{\mu-1}\Gamma\big(\frac{\mu-\nu+1}{2}\big)\Gamma\big(\frac{\mu+\nu+1}{2}\big)}t_{\mu,\nu}(x) \\
&=\frac{1}{2^{\mu+1}\Gamma\big(\frac{\mu-\nu+3}{2}\big)\Gamma\big(\frac{\mu+\nu+3}{2}\big)} {}_1F_2\bigg(1;\frac{\mu-\nu+3}{2},\frac{\mu+\nu+3}{2};\frac{x^2}{4}\bigg).
\end{align*}
The following integral formulas can be found \cite{gaunt24} and result from a simple deduction from an indefinite integral formula of \cite{r64}. For $\mu\geq\nu>-1/2$, $a>0$ and $x>0$,
\begin{align}\label{370}\int_0^x t^\mu K_\nu(at)\,\mathrm{d}t&=\frac{2^{\mu-1}}{a^{\mu+1}}\Gamma\bigg(\frac{\mu-\nu+1}{2}\bigg)\Gamma\bigg(\frac{\mu+\nu+1}{2}\bigg)G_{\mu,\nu}(ax), \\
\label{kint}\int_x^\infty t^\mu K_\nu(at)\,\mathrm{d}t&=\frac{2^{\mu-1}}{a^{\mu+1}}\Gamma\bigg(\frac{\mu-\nu+1}{2}\bigg)\Gamma\bigg(\frac{\mu+\nu+1}{2}\bigg)\tilde{G}_{\mu,\nu}(ax), 
\end{align}
where $G_{\mu,\nu}(x)$ and $\tilde{G}_{\mu,\nu}(x)$ are defined as in (\ref{gmn2}). With (\ref{370}) and (\ref{kint}) and a change of variables we obtain that, for $\mu\geq\nu/2>-1/2$, $a>0$ and $x>0$,
\begin{align}\label{37} \int_0^x t^\mu K_\nu(a\sqrt{t})\,\mathrm{d}t&=\frac{2^{2\mu+1}}{a^{2\mu+2}}\Gamma\Big(\mu-\frac{\nu}{2}+1\Big)\Gamma\Big(\mu+\frac{\nu}{2}+1\Big)G_{2\mu+1,\nu}(a\sqrt{x}), \\
\label{kint2}\int_x^\infty t^\mu K_\nu(a\sqrt{t})\,\mathrm{d}t&=\frac{2^{2\mu+1}}{a^{2\mu+2}}\Gamma\Big(\mu-\frac{\nu}{2}+1\Big)\Gamma\Big(\mu+\frac{\nu}{2}+1\Big)\tilde{G}_{2\mu+1,\nu}(a\sqrt{x}). 
\end{align}
In the case $\mu=\nu=0$, formula (\ref{kint2}) simplifies to
\begin{align}\label{kint3}\int_x^\infty  K_0(a\sqrt{t})\,\mathrm{d}t=\frac{2}{a}\sqrt{x}K_1(a\sqrt{x}), 
\end{align}
which follows readily from a straightforward direct integration using the standard formula $\frac{\mathrm{d}}{\mathrm{d}x}(xK_1(x))=-K_0(x)$ and the limiting form $K_{\nu}(x)\sim(\pi/(2x))^{1/2}\mathrm{e}^{-x}$, $x\rightarrow\infty$. 

\subsection{Hypergeometric and Meijer $G$-functions}

The \emph{Gaussian hypergeometric function} is defined, for $|x|<1$, by the power series
\begin{equation}
\label{gauss}
{}_2F_1(a,b;c;x)=\sum_{j=0}^\infty\frac{(a)_j(b)_j}{(c)_j}\frac{x^j}{j!},
\end{equation}
where $(u)_j=u(u+1)\cdots(u+j-1)$ is the ascending factorial.

The \emph{Whittaker function} is defined, for $x\in\mathbb{R}$, by
\begin{equation*}
W_{\kappa,\mu}(x)=\mathrm{e}^{-x/2}x^{\mu+1/2}U\bigg(\frac{1}{2}+\mu-\kappa,1+2\mu,x\bigg),    
\end{equation*}
where $U(a,b,x)$ is the confluent hypergeometric function of the second kind (see \cite[Chapter 13]{olver}).

The \emph{Fox--Wright} function is a generalisation of the generalised hypergeometric function $_{p}F_q$, which is defined, for $|x|<1$, by the power series
\begin{equation}
\label{fw} _{p}\Psi_q\bigg({(a_1,A_1),\ldots,(a_p,A_p) \atop (b_1,B_1),\ldots,(b_q,B_q)}\,\bigg|\,x \bigg)=\sum_{j=0}^\infty\frac{\Gamma(a_1+A_1j)\cdots\Gamma(a_p+A_pj)}{\Gamma(b_1+B_1j)\cdots\Gamma(b_q+B_qj)}\frac{x^j}{j!}.    
\end{equation} 

The \emph{Meijer $G$-function} is defined, for $x\in\mathbb{R}$, by the contour integral
\begin{equation}\label{mdef}G^{m,n}_{p,q}\bigg(x \, \bigg|\, {a_1,\ldots, a_p \atop b_1,\ldots,b_q} \bigg)=\frac{1}{2\pi \mathrm{i}}\int_L\frac{\prod_{j=1}^m\Gamma(b_j-s)\prod_{j=1}^n\Gamma(1-a_j+s)}{\prod_{j=n+1}^p\Gamma(a_j-s)\prod_{j=m+1}^q\Gamma(1-b_j+s)}x^s\,\mathrm{d}s,
\end{equation}
where the integration path $L$ separates the poles of the factors $\Gamma(b_j-s)$ from those of the factors $\Gamma(1-a_j+s)$. We use the convention that the empty product is $1$.

The Meijer $G$-function is symmetric in the parameters $a_1,\ldots,a_n$; $a_{n+1},\ldots,a_p$; $b_1,\ldots,b_m$; and $b_{m+1},\ldots,b_q$.  Thus, if one the $a_j$'s, $j=n+1,\ldots,p$, is equal to one of the $b_k$'s, $k=1,\ldots,m$, the $G$-function reduces to one of lower order.  For example,
\begin{equation}\label{lukeformula}G_{p,q}^{m,n}\bigg(x \; \bigg| \;{a_1,\ldots,a_{p-1},b_1 \atop b_1,\ldots,b_q}\bigg)=G_{p-1,q-1}^{m-1,n}\bigg(x \; \bigg| \;{a_1,\ldots,a_{p-1} \atop b_2,\ldots,b_q}\bigg), \quad m,p,q\geq 1.
\end{equation}
The $G$-function satisfies the relation
\begin{equation}\label{meijergidentity}x^\alpha G_{p,q}^{m,n}\bigg(x \; \bigg| \;{a_1,\ldots,a_p \atop b_1,\ldots,b_q}\bigg)=G_{p,q}^{m,n}\bigg(x \, \bigg| \,{a_1+\alpha,\ldots,a_p+\alpha \atop b_1+\alpha,\ldots,b_q+\alpha}\bigg).
\end{equation}

We have the following special cases of the $G$-function:
\begin{align}
\label{g300}G^{3,1}_{1,3} \bigg( x \bigg| { a+\frac{1}{2} \atop b,2a-b,a } \bigg)&=\frac{\pi^{5/2}x^a}{2\cos((b-a)\pi)}\big(J_{b-a}^2(\sqrt{x})+Y_{b-a}^2(\sqrt{x})\big),\\
\label{g30}G^{3,1}_{1,3} \bigg( x \bigg| { a+\frac{1}{2} \atop a+\frac{1}{2},-a,a } \bigg)&=\frac{\pi^2}{\cos(2\pi a)}\mathbf{K}_{2a}(2\sqrt{x}),\\
\label{g301}G^{3,1}_{1,3} \bigg( x \bigg| { a \atop b,a-\frac{1}{2},a} \bigg)&=\frac{\pi^2}{\sin((a-b)\pi)}x^{\frac{1}{4}(2a+2b-1)}\mathbf{K}_{a-b-1/2}(2\sqrt{x}),\\
\label{redm0}G^{4,0}_{0,4} \bigg( x \bigg| { - \atop a,a+\frac{1}{2},b,b+\frac{1}{2} } \bigg)&=4\pi x^{(a+b)/2}K_{2(a-b)}(4x^{1/4}), \\
\label{redm}G^{4,0}_{0,4} \bigg( x \bigg| { - \atop a, a+\frac{1}{2},2a-b,b } \bigg)&=8\sqrt{\pi} x^a K_{2(b-a)}( 2^{3/2} (-x)^{1/4}) K_{2(b-a)}((-1)^{-1/4} 2^{3/2}x^{1/4}),
\end{align}
where we obtained (\ref{g300}) by simplifying equation 6.5(39) of \cite{luke} using the identity (\ref{jhjh}), we used (\ref{boldk}) to express equation 6.5(37) of \cite{luke} in terms of the Struve function $\mathbf{K}_\nu$, and the formula (\ref{g301})
can be found at http://functions.wolfram.com/07.34.03.0980.01.

Combining equations 5.4(1) and 5.4(13) of \cite{luke} yields the indefinite integral formula
\begin{equation}\label{mint}\int x^{\alpha-1}G_{p,q}^{m,n}\bigg(x \; \bigg| \;{a_1,\ldots,a_p \atop b_1,\ldots,b_q}\bigg)\,\mathrm{d}x=x^\alpha G_{p+1,q+1}^{m,n+1}\bigg(x \, \bigg| \,{1-\alpha,a_1,\ldots,a_p \atop b_1,\ldots,b_q,-\alpha}\bigg).
\end{equation}

\subsection{Exponential, sine and cosine integrals}
The \emph{exponential, sine and cosine integrals} are defined, for $|\mathrm{Arg}(z)|<\pi$, by 
\begin{align*}
E_1(z)=\int_z^\infty\frac{\mathrm{e}^{-t}}{t}\,\mathrm{d}t, \quad \mathrm{Si}(z)=\int_0^z\frac{\sin t}{t}\,\mathrm{d}t, \quad \mathrm{Ci}(z)=-\int_z^\infty\frac{\cos t}{t}\,\mathrm{d}t.  
\end{align*}
For $x\in\mathbb{R}$,
\begin{equation}\label{ci} E_1(-\mathrm{i}x)=\mathrm{Ci}(x)-\mathrm{i}\bigg(\mathrm{Si}(x)+\frac{\pi}{2}\bigg).
\end{equation}
A specialisation of equation 6.614(4) of \cite{gradshetyn} yields the integral formula: for $a>0$, $t>0$,
\begin{equation}\label{thn}
\int_0^\infty \mathrm{e}^{\mathrm{i}tx}K_0(2\sqrt{ax})\,\mathrm{d}x=\frac{1}{2t}\mathrm{e}^{a\mathrm{i}/t}\bigg(\pi-\mathrm{i}E_1\Big(-\frac{a\mathrm{i}}{t}\Big)\bigg).   
\end{equation}
More generally, equation 6.643(3) of \cite{gradshetyn} states that: for $\mu+\nu>-1$, $\mu-\nu>-1$, $a>0$, $t\in\mathbb{R}$,
\begin{align}\label{thnaaa}
\int_0^\infty \!\mathrm{e}^{\mathrm{i}tx}x^{\mu/2-1/2}K_\nu(2\sqrt{ax})\,\mathrm{d}x=\frac{(-\mathrm{i}t)^{-\mu/2}}{2\sqrt{a}}\Gamma\Big(\frac{1+\mu+\nu}{2}\Big)\Gamma\Big(\frac{1-\nu+\mu}{2}\Big)\exp\Big(\frac{a\mathrm{i}}{2t}\Big)W_{-\mu/2,\nu/2}\Big(\frac{a\mathrm{i}}{t}\Big).
\end{align}
From (\ref{thn}), together with the formula $\cos x=(\mathrm{e}^{\mathrm{i}x}+\mathrm{e}^{-\mathrm{i}x})/2$,
Euler's formula $\mathrm{e}^{\mathrm{i}x}=\cos x+\mathrm{i}\sin x$ and equation (\ref{ci}), we deduce the following integral formula: for $a>0$, $t>0$,
\begin{align}\label{thn0}
\int_0^\infty \cos(tx)K_0(2\sqrt{ax})\,\mathrm{d}x&=\frac{1}{2t}\bigg\{\bigg(\frac{\pi}{2}-\mathrm{Si}\Big(\frac{a}{t}\Big)\bigg)\cos\Big(\frac{a}{t}\Big)+\mathrm{Ci}\Big(\frac{a}{t}\Big)\sin\Big(\frac{a}{t}\Big)\bigg\}.
\end{align}

\section*{Acknowledgements}
We would like to thank the reviewers for their helpful comments. RG is funded in part by EPSRC grant EP/Y008650/1. SL is supported by a University of Manchester Research Scholar Award.

\end{document}